\documentclass[11pt]{article}
\usepackage{geometry}                
\usepackage{graphicx}
\usepackage{amsmath,amssymb,amsthm,amsrefs}
\usepackage{epstopdf}

\usepackage{mdframed}
\usepackage{xcolor}
    \usepackage{xcolor}
    \definecolor{darkgreen}{cmyk}{1,0,1,.2}
    \definecolor{m}{rgb}{1,0.1,1}

\usepackage{setspace}
\usepackage{appendix}

\theoremstyle{definition}

\usepackage{tikz-cd}
\usepackage{url}
\theoremstyle{plain}
  \newtheorem{thm}{Theorem}[section]
  \newtheorem{lem}[thm]{Lemma}
  
  \newtheorem{cor}[thm]{Corollary} 
\theoremstyle{definition}
  \newtheorem{defn}[thm]{Definition}
  \newtheorem{rmk}[thm]{Remark}
  \newtheorem{ex}[thm]{Example}

\theoremstyle{plain}
  
\DeclareMathOperator{\im}{im}
\DeclareMathOperator{\coker}{coker}
\DeclareMathOperator{\rk}{rank}
\DeclareMathOperator{\tC}{Cone}
\DeclareMathOperator{\ttC}{\widetilde{Cone}}
\DeclareMathOperator{\wtC}{\widehat{Cone}}
\def\w{\wedge}
\def\om{\omega}
\def\Om{\Omega}

\def\mP{{\cal{P}}}
\def\com{c(\psi)}

\def\cpsi{c(\psi)}
\def\del{\partial}
\def\CM{\mathcal{M}}

\def\txw{r}
\def\da{d_A}
\def\db{d_B}
\def\dc{d_C}
\def\vp{\varphi}
\def\dap{d^\pi_A}
\def\itwo{\iota_2}
\def\vpt{\vp_2}
\def\dcc{d_{\widetilde{C}}}
\def\dcw{d_{\widehat{C}}}
\def\pio{\pi_1}
\def\tila{\tilde{a}}
\numberwithin{equation}{section}

\begin{document}

\title{
\bf\
{Mapping Cone and Morse Theory}
}

\author{David Clausen, Xiang Tang and Li-Sheng Tseng\\
\\
}

\date{}

\maketitle

\begin{abstract}

On a smooth manifold, we associate to any closed differential form a mapping cone complex.  The cohomology of this mapping cone complex can vary with the de Rham  cohomology class of the closed form.  We present a novel Morse theoretical description for the mapping cone cohomology.  Specifically, we introduce a Morse complex for the mapping cone complex which is generated by pairs of critical points with the differential defined by gradient flows and an integration of the closed form over spaces of gradient flow lines.  We prove that the cohomology of our cone Morse complex is isomorphic to the mapping cone cohomology and hence independent of both the Riemannian metric and the Morse function used to define the complex.  We also obtain sharp inequalities that bound the dimension of the mapping cone cohomology in terms of the number of Morse critical points and the properties of the specified closed form. Our results are widely applicable, especially for any manifold equipped with a geometric structure described by a closed differential form.  We also obtain a bound on the difference between the number of Morse critical points and the Betti numbers.

\end{abstract}

\tableofcontents


\section{Introduction}

Manifolds equipped with a geometric structure described by a closed differential form are widely studied. For example, a large class are the symplectic manifolds which by definition contain a non-degenerate, closed two-form.  Another large class are special holonomy manifolds.  These include Calabi-Yau threefolds and $G_2$ manifolds which carry a closed invariant three-form.  And even for complex manifolds that are non-K\"{a}hler, there often is a distinguished closed form.  For instance, on a complex hermitian threefold that satisfies the balanced condition, the square of the hermitian form is by definition a closed four-form.

For a manifold $(M, \psi)$, where $\psi\in \Om^\ell(M)$ is a $d$-closed $\ell$-form that represents a geometric structure, we seek invariants that are dependent on $\psi$.  Certainly, without $\psi$ being present, there is the well-known de Rham differential graded algebra $(\Om^*, d, \w)$ that results in basic invariants, i.e.~de Rham cohomology ring and Massey products, for any smooth manifold $M$.  So we are led to a simple question: Is there a natural extension of the de Rham algebra that can incorporate the additional geometric structure $\psi$?  In this paper, we begin by first pointing out that there is a mapping cone complex of differential forms that provides invariants that generally depends on $\psi$.  Though the notion of a mapping cone complex is widely used in homological algebra (see, for example, \cite{Weibel}), its use in the context of differential forms on general smooth manifolds, has not been much studied. 

\subsection{Mapping cone complex and cohomology of differential form}

We will motivate the construction of a mapping cone complex (or just called, the ``cone" complex) of differential forms as follows.   For differential forms, the exterior (or wedge) product is a natural operation.  Given a distinguished $d$-closed $\ell$-form $\psi$, we can think of $\psi$ not just as an element of $\Om^\ell(M)$, but additionally, as an operator or a ``map" on the space of differential forms via the wedge product, $\psi \w : \Om^k(M)\to \Om^{k+\ell}(M)$.  

Certainly, considering $\psi$ as an operator can be helpful.  For combining $\psi\w$ with the exterior differential $d$ give what is commonly called the ``twisted" differential, $d+\psi\w$. But there are two notable drawbacks if we were to consider  $d+\psi\w$ as the differential of a complex: (1) $d+\psi\w$ generally does not square to zero unless $\ell$ is odd; (2) $d+\psi\w$ also does not preserve grading as it maps $\Om^k(M)$ to the mixed degree sum $\Om^{k+1}(M) \oplus \Om^{k+\ell}(M)$.  

We can however give a simple solution to address these two issues.  For we can consider mapping both $d$ and $\psi \w$ at the same time into $\Om^{k+1}(M)$.  This would require changing the domain to be the direct sum $\Om^k(M) \oplus \Om^{k-\ell+1}(M)$ and mapping into $\Om^{k+1}(M)$ as follows:
\begin{equation*}
\begin{tikzcd}
\Om^k(M) \arrow[r, "d"]  & \Om^{k+1}(M)\\
&  \ \Om^{k-\ell+1}(M)\,.  \arrow[u, swap, "~\psi\, \w"]
\end{tikzcd}
\end{equation*}
This is suggestive of defining what we shall call the space of \textit{cone} forms, $\tC^\bullet(\psi)= \Om^\bullet(M) \oplus \Om^{\bullet-\ell+1}(M)$, which consists of \textit{pairs} of differential forms.  And with it, we introduce the cone differential:
\begin{align}\label{Cdef}
d_C: \tC^k(\psi) &\longrightarrow \tC^{k+1}(\psi)\\
\begin{pmatrix} \Om^k \\ \Om^{k-\ell+1} \end{pmatrix} &\longmapsto \begin{pmatrix} d & \psi\,\w  \\ 0 & (-1)^{\ell-1} d \end{pmatrix} \begin{pmatrix} \Om^{k} \\ \Om^{k-\ell+1} \end{pmatrix} = \begin{pmatrix} d\, \Om^k + \psi \w \Om^{k-\ell+1} \\ (-1)^{\ell-1} d\,\Om^{k-\ell+1} \end{pmatrix}. \nonumber
\end{align}
It is straightforward to check that $d_C\, d_C  = 0$.  And this immediately results in the following cone cohomology defined with respect to any $d$-closed form, $\psi\in \Om^\ell(M)$:
\begin{align}\label{mccoh}
H^k(\tC(\psi))  = \dfrac{\ker d_C \cap \tC^k(\psi)}{\im d_C \cap \tC^k(\psi)}\,.
\end{align}

Let us emphasize that this cone cohomology is not a topological invariant and generally depends on $\psi$.  In fact, the cohomology essentially contains the product structure information of the de Rham cohomology in relation to $[\psi]\in H^\ell_{dR}(M)$.  To see this, let us express $\tC^k(\psi)= \Om^k(M) \oplus \Om^{k-\ell+1}(M)$ in terms of an exact sequence
\begin{equation*}
\begin{tikzcd}
0 \arrow[r]  &
\Omega^k(M) \arrow[r,"\iota_{dR}"]  &
\tC^k(\psi)  \arrow[r,"\pi_{dR}"]  & \Omega^{k-\ell+1}(M) \arrow[r] & 0\,,
\end{tikzcd}
\end{equation*}
where $\iota_{dR}$ is the inclusion map and $\pi_{dR}$ is the projection onto the second component.  
This short exact sequence standardly leads to a long exact sequence of cohomologies
\begin{equation}
\begin{tikzcd} 
\ldots \arrow[r] & H_{dR}^{k-\ell} \arrow[r, "\text{[}\psi\text{]}"] & H_{dR}^{k} \arrow[r, 
"\text{[}\iota_{dR}\text{]}"] & H^{k}(\tC(\psi)) \arrow[r, "\text{[}\pi_{dR}\text{]}"] & H_{dR}^{k-\ell+1} \arrow[r, "\text{[}\psi\text{]}"] & H_{dR}^{k+1}\arrow[r] &\dots
\end{tikzcd}
\end{equation}
which implies that 
\begin{align}\label{hckk}
H^{k}(\tC(\psi))\cong \coker(\,[\psi]\!: H_{dR}^{k-\ell}\to H_{dR}^k\,)\,\oplus\,\, \ker(\,[\psi]\!: H_{dR}^{k-\ell+1}\to H_{dR}^{k+1}\,)\,.
\end{align}
Hence, we see that $H^k(\tC(\psi))$ encodes the product structure of the de Rham cohomology under the linear map $[\psi]: H_{dR}^\bullet \to H_{dR}^{\bullet + \ell}$.   And in general, the kernel and cokernel of such a map can vary as $[\psi]$ varies in de Rham cohomology.

Let us make here two observations.   First, the cone cohomology has appeared previously in the special case of a symplectic manifold $(M^{2n}, \om)$ where $\psi$ was specified to be $\psi=\om^{p+1}$, for $p=0, \ldots, n-1$.   In this special setting, Tanaka-Tseng \cite{TT} showed that 
$H(\tC(\om^{p+1}))$ is isomorphic to the symplectic cohomologies of differential forms called $p$-filtered cohomologies, $F^pH(M, \om)$, introduced by Tsai-Tseng-Yau \cite{TTY} and the underlying algebras of these two cohomologies are quasi-isomorphic as $A_\infty$ algebras.  Of interest, the dimensions of $F^pH(M, \om)$ have been shown to vary with the de Rham class of the symplectic structure in various examples: a six-dimensional symplectic nilmanifold \cite{TY2}, a three-torus product with a three-ball, $T^3 \times B^3$ \cite{TW}, and for classes of open symplectic 4-manifolds that are homeomorphic but not diffeomorphic \cite{GTV}.  By the isomorphism of $F^pH(M,\om)\cong H(\tC(\om^{p+1}))$, we know that $H(\tC(\om^{p+1}))$ would also vary with $[\om]$ in these examples.  

Our second observation is especially worthy to emphasize.  It is that the cone cohomology defined above only requires $\psi$ to be $d$-closed and nothing more.  Geometric structures such as a symplectic two-form or an associative three-form on $G_2$ manifolds often have an additional property like non-degeneracy besides $d$-closedness.  Certainly, when $\psi$ has more properties besides being closed, the cone complex may have additional structures as well.  For instance, if $\psi$ is an even-degree form and additionally an element of the integral cohomology $H^\ell(M, \mathbb{Z})$, then the cone cohomology interestingly can be interpreted geometrically as the de Rham cohomology of an $S^{\ell-1}$ sphere bundle over $M$ with Euler class given by $\psi$ \cite{TT}.  
Though we may have been initially motivated to seek invariants with respect to a geometric structure, with the cone complex at hand, it is useful to study the cone cohomology without imposing any additional condition on $\psi$.  
Below we shall turn our focus to developing a Morse theory for the cone cohomology $H(\tC(\psi))$ defined in \eqref{mccoh} with respect to any $d$-closed form, $\psi$.

\subsection{Morse complex for the mapping cone}

On a Riemannian manifold, the de Rham cohomology can be described alternatively as the cohomology of a Morse complex (or also referred to as the Morse-Witten or Smale-Thom complex).  Besides the Riemannian metric $g$, to define a Morse complex requires the introduction of a special function $f$ on $M$, called a Morse function, which is defined by the property that the Hessian at each critical point is non-degenerate.  The elements of the Morse complex $C^k(M,f)$ are then generated by the critical points of $f$, $q\in Crit(f)$, and grouped together by their index, $k=ind(q)$, the number of negative eigenvalues of the Hessian matrix at $q$. The differential of the complex, $\partial$, is defined by the gradient flow, $-\nabla f$, from one critical point to another.  Explicitly, in local coordinates $\{x^i\}$, the gradient flow is $\dot{x}^i(t) = -g^{ij}\partial f/\partial x^j$ which involves the Riemannian metric $g$. (We shall assume throughout this paper that the metric $g$ satisfies the usual Smale transversality condition, that is, the submanifolds that flow into or from the critical points are transverse.)  The resulting cohomology of the Morse cohomology is famously known to be isomorphic to the standard cohomology, and therefore, the Morse cohomology generally does not depend on the choice of the Morse function $f$ and  metric $g$ that are used to define it.  As a corollary of this isomorphism, there are the renown Morse inequalities which bound the Betti numbers of $M$ in terms of the number of critical points of the Morse function.

Now, for a smooth manifold equipped with a geometric structure described by a closed $\ell$-form $(M, \psi)$, we have discussed the cone cohomology $H(\tC(\psi))$ which provides basic geometrical invariants that are dependent on $[\psi]\in H^\ell_{dR}(M)$.  Given the connection between the mapping cone complex and the de Rham complex, it is natural to ask whether there is also a Morse theory-type description for the mapping cone cohomology?   Such a Morse description would necessarily require the involvement of $\psi$ in some intrinsic way.  And if a Morse theory for $(M, \psi)$ exists, can we bound the dimensions of the cone cohomology by means of the critical points of a Morse function and their gradient flows?

In this paper, we answer both questions in the affirmative.  

Motivated by the relationship between de Rham complex and the Morse cochain complex over $\mathbb{R}$ (see Table \ref{Tab0}), we define in the following a {\it cone} Morse complex also over $\mathbb{R}$.  
\begin{table}
\begin{center}
\begin{tabular}{ c | c | c |c  } 
Complex & Cochains & Differential &Cohomology \\
\hline
\hline
de Rham & $\Omega^\bullet(M)$ & $d$ & $H^\bullet_{dR}(M)$\\
\hline
Morse & $C^\bullet(M, f)$ & $\partial$  & $H^\bullet_{C(f)}(M)$\\
\hline\hline
Cone & $\Omega^\bullet(M)\oplus\Omega^{\bullet-\ell+1}(M)$ & $d_C =\begin{pmatrix} d& \psi\\ 0 & (-1)^{\ell-1}d\end{pmatrix}$&$H^\bullet(\tC(\psi))(M)$\\
\hline
Cone Morse & $C^\bullet(M, f) \oplus C^{\bullet-\ell+1}(M, f)$ &
$\partial_C=\begin{pmatrix} \partial & c(\psi) \\ 0 & (-1)^{\ell-1}\partial \end{pmatrix}$&$H^\bullet(\tC(c(\psi)))(M)$\\
\hline
\end{tabular} 
\end{center}
\caption{The relations between the de Rham and Morse cochain complexes and Cone and Cone-Morse complexes.}\label{Tab0}
\end{table}
\begin{defn}\label{Ccdef}
Let $(M, g)$ be an oriented, Riemannian manifold and $f$ a Morse function satisfying the Morse-Smale transversality condition.  Let $C^k(M,f)$ be the $\mathbb{R}$-module with generators the critical points of $f$ with index $k$.  Given a $d$-closed form $\psi\in \Om^\ell(M)$, we define the {\bf cone Morse cochain complex} of $\psi$, $\tC(c(\psi))=(C^\bullet(M,f)\oplus C^{\bullet-\ell+1}(M,f),\, \del_C)$:
\begin{equation*}
\begin{tikzcd}
\ldots \arrow[r, "\del_C"]  &
C^{k}(M,f)\oplus C^{k-\ell+1}(M,f) \arrow[r,"\del_C"] & C^{k+1}(M,f)\oplus C^{k-
\ell+2}(M,f) \arrow[r,"\del_C"]  & \ldots
\end{tikzcd}
\end{equation*}
with 
\begin{align}\label{delCdef}
\del_C= \begin{pmatrix} \del & c(\psi) \\ 0 & (-1)^{\ell-1}\del \end{pmatrix}\,.
\end{align}
Here, $\partial$ is the standard Morse cochain differential defined by gradient flow, and $c(\psi): C^k(M,f)\to C^{k+\ell}(M,f)$ acting on a critical point of index $k$ is defined to be
\begin{align}\label{cpsidef}
c(\psi)\,q_{k}=\sum_{ind(r)={k+\ell}}\left( \int_{\overline{\CM(r_{k+\ell}, q_k)}}\, \psi \right)r_{k+\ell}
\end{align}
where $\CM(r_{k+\ell}, q_k)$ is the $\ell$-dimensional submanifold of $M$ consisting of all flow lines from the index $k+\ell$ critical point, $r_{k+\ell}\,$, to $q_k$.
\end{defn}
\begin{table}
\begin{center}
\begin{tabular}{ c | c | c  } 
& de Rham  &  Morse  \\
\hline
Cochains & 
$\Omega^\bullet(M)$ & $C^\bullet(M, f)$ \\ 
\hline
 Differential & $d$ & $\partial$ (gradient flow)  \\
 \hline
Cohomology &   
\multicolumn{2}{c}{$H_{dR}^k(M) \cong H^k_{C(f)}(M)$}  \\
\hline
Morse & 
\multicolumn{2}{c}{$b_k \leq m_k$}  \\
Inequalities & 
\multicolumn{2}{c}{$\displaystyle\sum_{k=0}^{j} (-1)^{j-k}b_k \leq \sum_{k=0}^{j} (-1)^{j-k}m_k$} \\
\hline
\end{tabular} \\
\end{center}
\caption{The relations between the de Rham and Morse cochain complexes and Morse inequalities, where $b_k = \dim H^k_{dR}(M)$ and $m_k = \dim C^k(M, f)$}\label{Tab11}
\end{table}
\begin{table}
\begin{center}
\begin{tabular}{c| c | c } 
 &  Cone  & Cone Morse \\
\hline
Cochains & $
\Omega^\bullet(M)\oplus\Omega^{\bullet-\ell+1}(M)$& $C^\bullet(M, f) \oplus C^{\bullet-\ell+1}(M, f)$\\ 
\hline

Differential & $d_C =\begin{pmatrix} d& \psi\\ 0 & (-1)^{\ell-1}d\end{pmatrix}$ &   $\partial_C=\begin{pmatrix} \partial & c(\psi) \\ 0 & (-1)^{\ell-1}\partial \end{pmatrix}$ \\
 \hline
Cohomology &   \multicolumn{2}{c}{ $
H^k(\tC(\psi))(M) \cong H^k(\tC(c(\psi)))(M)$} 
\\
\hline
Cone-Morse & \multicolumn{2}{c}{ $b_k^\psi 
\leq  m_{k}-v_{k-\ell}+m_{k-\ell+1}-v_{k-\ell+1}$} \\
Inequalities   &  \multicolumn{2}{c}{$\displaystyle\sum_{k=0}^{j} (-1)^{j-k}b^\psi_k \leq\displaystyle\sum_{k=0}^{j} (-1)^{j-k}(m_k - v_{k-\ell}+m_{k-\ell+1}-v_{k-\ell+1})$} \\
\hline
\end{tabular}
\end{center}
\caption{The relations between the cone complex and the cone Morse complex and  inequalities in the presence of a closed $\ell$-form $\psi$, where $b^\psi_k=\dim H^k(\tC(\psi))(M)$ and $v_k=\rk [\cpsi: C^k(M, f) \to C^{k+\ell}(M,f)]$. }
\label{Tab22}
\end{table}

Notice that the elements of the Morse cone complex, $\text{Cone}^k(c(\psi))=C^k(M,f)\oplus C^{k-l+1}(M,f)$, can be generated by {\it pairs} of critical points, of index $k$ and $k-l+1$.  The differential $\del_C$ consists of the standard Morse differential $\partial$ from gradient flow coupled with the $c(\psi)$ map which involves an integration of $\psi$ over the space of gradient flow lines.  This $c(\psi)$ map has appeared in Austin-Braam \cite{AB} and Viterbo \cite{Viterbo} to define a cup product on Morse cohomology.  It satisfies the following Leibniz rule
\begin{align}\label{MLeibniz}
\del c(\psi)+(-1)^{\ell+1}c(\psi)\del=-c(d\psi)\,.
\end{align}
A check of the $\pm$ signs of this equation together with a description of the orientation of $\overline{\CM(r_{k+\ell+1}, q_k)}$ is given in Appendix A.  With \eqref{MLeibniz} and $\del \del=0\,,$ they together imply $\del_C\,\del_C=0$.  

We will prove in Section 2 that the cohomology of our cone Morse complex $\tC(c(\psi))$ is isomorphic to the cohomology of the cone complex $\tC(\psi)$ of differential forms.  
\begin{thm}\label{MIso}
Let $M$ be a closed, oriented Riemannian manifold and $\psi\in \Om^l(M)$ a $d$-closed form.  There exists a chain map $\mP_{C}:(\tC^\bullet(\psi), d_C)\to (\tC^\bullet(c(\psi)), \del_C)$ that is a quasi-isomorphism, and therefore, for any $k\in  \mathbb{Z}$,
\begin{align*}
H^k(\tC(\psi)) \cong H^k(\tC(c(\psi)))\,.
\end{align*}
\end{thm}
Theorem \ref{MIso} importantly shows that the cohomology of the cone Morse complex is independent of the choice of both the Morse function $f$ and the Riemannian metric $g$ used to define $\tC(c(\psi))$.  It is also worthwhile to emphasize that the above theorem is a general one, applicable for any closed smooth manifold, odd or even dimensional, with respect to any closed differential form on the manifold.

Having obtained a cone Morse theory, we would like to write down the Morse-type inequalities that bounds the dimension of the cone cohomology which we will denote by $b^\psi_k = \dim H^k(\tC(\psi))$.  Specifically, we would like to bound the $b^\psi_k$'s by the properties of the Morse functions.  Recall that the standard Morse inequalities (for a reference, see e.g.~\cite{Milnor}) bounds the $k$-th Betti number $b_k = \dim H_{dR}^k(M)$ by $m_k$, the number of index $k$ critical points of a Morse function.  The usual Morse inequalities can be stated concisely as the existence of a polynomial $Q(t)$ with non-negative integer coefficients such that 
\begin{align}\label{sm1}
\sum_{k=0} m_k\, t^k = \sum_{k=0} b_k\, t^k \,+\, (1+t)\, Q(t)\,.
\end{align}
This is equivalent to what is called the strong Morse inequalities 
\begin{align}\label{sm2}
\sum_{i=0}^k (-1)^{k-i}\,b_i \, \leq \,\sum_{i=0}^k (-1)^{k-i}\,m_i\,,\qquad k=0, \ldots, \dim M
\end{align}
which imply the weak Morse inequalities
\begin{align}\label{wm1}
b_k \leq m_k\,,
\end{align}
also for $k=0, \ldots, \dim M$.  

We can derive the analogous Morse-type inequalities results for the cone cohomology.   We obtain the following:
\begin{thm}\label{CMineq}
Let $(M, \psi, f, g)$ be a closed, oriented Riemannian manifold with Morse function $f$, Riemannian metric $g$, and $\psi\in\Om^\ell(M)$ a $d$-closed form.   Then
there exists a polynomial $Q(t)$ with non-negative integer coefficients such that
\begin{align*}
(1+t^{\ell-1})\sum_{k=0} m_k\, t^k - (t^{\ell-1}+t^{\ell})\sum_{k=0} v_k\, t^k= \sum_{k=0} b^{\psi}_k\, t^k \, +\, (1+t)\, Q(t) 
\end{align*}
where $b^\psi_k=\dim H^k(\tC(\psi))$ and 
\begin{align}\label{Ivkdef}
v_k=\rk\left(\cpsi:C^{k}(M,f) \to C^{k+\ell}(M,f)\right)\,.
\end{align}
Equivalently, we have the following inequalities:
\begin{itemize}
\item[(A)] Weak cone Morse inequalities 
\begin{align}\label{Iwcm}
b_k^{\psi} \ \leq \,
m_k-v_{k-\ell}+m_{k-\ell+1}-v_{k-\ell+1}\,, \quad\quad  k = 0, \dots,\dim M + \ell -1;
\end{align}
\item[(B)] Strong cone Morse inequalities
\begin{align}\label{Iscm}
\sum_{k=0}^j (-1)^{j-k}b^\psi_k \leq 
\sum_{k=0}^j (-1)^{j-k} (m_k-v_{k-\ell}+m_{k-\ell+1}-v_{k-\ell+1}) \,,
\end{align}
for $ j = 0, \dots,\dim M + \ell -1\,$.
\end{itemize}
\end{thm}

\medskip

For the cone Morse inequalities \eqref{Iwcm}-\eqref{Iscm}, it is worth pointing out that the dimension $b^\psi_k$ on the left-hand-side is dependent only on the class $[\psi]\in H^\ell_{dR}(M)$ as evident from \eqref{hckk}.  However, the right-hand-side, in particular the $v_k$'s, has dependence on $\psi$, as a differential form, and not just on the de Rham class $[\psi]$.  (A simple example of the cone Morse bound varying when $\psi$ is varied within a fixed cohomology class is give in Example \ref{ex3}.)  Hence, it is worthwhile in the context of cone Morse inequalities to consider $\psi$ as a $d$-exact form.  In the special case where $\psi$ is a $d$-exact two-form, the strong cone Morse inequalities of \eqref{Iscm} turn out to imply an interesting bound for the difference between the number of critical points of a Morse function and the Betti numbers.  
\begin{cor}\label{bmv}
For $\psi\in \Om^2(M)$ a $d$-exact form, we have the following bounds for $k=1, \ldots, \dim M $, 
\begin{align}\label{bmvIneq}
b_k \leq m_k - v_{k-1}\,. 
\end{align}
\end{cor}
This corollary gives an estimate for the difference between $m_k$ and $b_k$ in terms of $v_{k-1}$, which involves integrating $\psi$ over gradient flow spaces.  It represents an improvement to the classical Morse inequality \eqref{wm1} and can be used to quickly determine whether a Morse function is perfect or not.  (Recall a perfect Morse function is one where $m_k=b_k$ for all $k$.)  In considering the inequality \eqref{bmvIneq}, note that it applies for any exact two-form $\psi = d\alpha$ where $\alpha \in \Om^1(M)$.  Then $v_{k-1}$ can only be non-zero if at least one moduli space of flow lines $\overline{\CM(r_{k+1}, q_{k-1})}$, which is the domain of integration in \eqref{cpsidef}, has a boundary.  For if one $\partial{\overline{\CM}}\neq 0$, then we can choose to work with a one-form $\alpha$ that takes value only along a small localized region along the boundary, such that the boundary integral of $\alpha$ is non-zero and thus generates $v_{k-1}>0$.  

Of course, not all manifolds have perfect Morse functions.  For instance, Morse functions on manifolds that has torsion in its homology class must satisfy the inequalities \cite{Pitcher}
\[ b_k + \mu_k + \mu_{k-1}   \leq  m_k\]
where $\mu_k$ is the minimum number of generators of the torsion components of $H_k(M,\mathbb{Z})$.  Hence, our results imply that a manifold with torsion must have a moduli space of flow lines $\overline{\CM(r_{k+1}, q_{k-1})}$ that has a non-trivial boundary.

The outline of this paper is as follows.  In Section 2, we define our cone Morse complex in detail and proof the isomorphism between the cone and cone Morse cohomology of Theorem \ref{MIso}.  In Section 3, we study the implication of the isomorphism and derive the cone Morse inequalities of Theorem \ref{CMineq}, and also Corollary \ref{bmv}.  In addition, we show that when $f$ is a perfect Morse function, the cone Morse inequalities \eqref{Iwcm}-\eqref{Iscm} become equalities.  In Section 4, we demonstrate the various properties of the cone Morse cohomology and inequalities in the simple, yet rich example of the two-sphere $S^2$.  In Appendix A, we gather our conventions for defining Morse theory and carefully prove the Leibniz rule relation given in \eqref{MLeibniz}.  In Appendix B, we review the various complexes that can arise when there exists a chain map between cochain complexes.  And in Appendix C, we briefly visit the question of giving a differential graded algebra (DGA) structure for the cone complex.  We describe a natural DGA for the cone complex in the case when $\psi$ is an even degree form and introduce a modified DGA when $\psi$ is an odd form.   Lastly, we mention that in a separate companion paper to this work \cite{CTT2}, we study some analytic aspects of the cone complex  in the special case where $\psi$ is a symplectic structure.  We present there an analysis of the symplectic cone Laplacian operator and used the analytic Witten deformation method to study the symplectic cone Morse theory.

\



\noindent{\it Acknowledgements.~} 
We thank Poom Lertpinyowong, Yu-Shen Lin, Daniel Morrison, Richard Schoen, Daniel Waldram, and Jiawei Zhou for helpful discussions.  The second author was supported in part by NSF Grants DMS-1800666 and DMS-1952551.  The third author would like to acknowledge the support of the Simons Collaboration Grant No. 636284.


\section{Cone Morse complex: Cone(c($\psi$))}

\subsection{Preliminaries: Morse complex and $c(\psi)$}
To begin, let $f$ be a Morse function and $g$ a Riemannian metric on $M$. We will assume throughout this paper that  $(f,g)$ satisfy the standard Morse-Smale transversality condition.  The elements of the Morse cochain complex $C^\bullet(M,f)$ are $\mathbb{R}$-modules with generators critical points of $f$, graded by the index of the critical points, with boundary operator $\del$ determined by the counting of gradient lines, i.e 
$$\del q_k=\displaystyle \sum_{ind(r)=k+1}n(r_{k+1}, q_k)\,r_{k+1}
$$
where $n(r_{k+1}, q_k)=\#\widetilde{\CM}(r_{k+1}, q_k)$ is a count of the moduli space of gradient flow lines with orientation modulo reparametrization. 

Note that Morse theory is typically presented as a homology theory, and hence, flowing from index $k$ to index $k-1$ critical points. To match up with the cochain complex of differential forms, we here work with the dual Morse cochain complex.  Hence, our $\del$ is the adjoint of the usual Morse boundary map under the inner product $\langle q_{k_i}, q_{k_j}\rangle = \delta_{ij}$. 

Following Austin-Braam \cite{AB} and Viterbo \cite{Viterbo}, we define 
\begin{align*}
c(\psi)q_k=\displaystyle \sum_{ind(r)=k+\ell}\left(\int_{\overline{\CM(r_{k+\ell}, q_{k})}} \psi\right)r_{k+\ell}
\end{align*} 
where $\psi \in \Omega^\ell(M)$ is an $\ell$-form and $\CM(r_{k+\ell}, q_k)$ is the submanifold of all points that flow from $r_{k+\ell}$ to $q_k$, oriented as in \cite{AB}. 
From Appendix A equation \eqref{Alemma}, we have the Leibniz-type product relation 
\[\del c(\psi)+(-1)^{\deg(\psi)+1}c(\psi)\del=-c(d\psi)\]
specifying a sign convention that is ambiguous in Austin-Braam \cite{AB} and Viterbo \cite{Viterbo}. 
 Thus, for instance, for $\psi\in \Om^\ell(M)$ a $d$-closed form, we have the relation 
 \begin{align}
 \del \com=(-1)^\ell \com\del\,.
 \end{align}


\subsection{Chain map between $\tC(\psi)$ and $\tC(\cpsi)$}
As explained by Bismut, Zhang and Laudenbach \cites{BZ, Zhang}, there is a chain map $\mP:\Omega^k(M) \to C^k(M,f)$ between the de Rham complex and the Morse cochain complex given by
\[ \mP\phi= \sum_{q_k \in Crit(f)} \left( \int_{\overline{U_{q_k}}} \phi\right)q_k \]
where $\phi\in \Om^k(M)$ and  $U_q$ is the set of all points on a gradient flow away from $q$. Being a chain map, 
\begin{align}\label{Pchain}
\del\, \mP=\mP\, d\,.
\end{align}
Bismut, Zhang and Laudenbach, in \cite{BZ}*{Theorem 2.9} (see also  \cite{Zhang}*{Theorem 6.4}), also proved the following: 
\begin{align}\label{Pisom}
\mP:H^k_{dR}(M) \to H^k_{C(f)}(M) \text{ is an isomorphism.} 
\end{align}
Furthermore, Austin-Braam \cite{AB}*{Section  3.5} showed that 
$\mP( \psi \wedge \gamma)$ and $c(\psi)\mP \gamma$  are cohomologous: 
\begin{align}\label{c()coh}
[\mP][\psi]=[c(\psi)][\mP]\,.
\end{align}
Motivated by these results, we wish to find an analogous chain map relating $\text{Cone}(\psi)= (\Om^\bullet(M) \oplus \, \Om^{\bullet-\ell+1}(M),\, d_C)$ with  $\tC(\cpsi)=(C^\bullet(M,f)\oplus C^{\bullet-\ell+1}(M,f),\, \del_C)$, where as given in \eqref{Cdef} and Definition \ref{Ccdef},
\begin{align*}
d_C&: \Om^k(M)\oplus  \Om^{k-\ell+1}(M)\to \Om^{k+1}(M)\oplus \Om^{k-\ell+2}(M)\\
\del_C&: C^k(M,f)\oplus C^{k-\ell+1}(M,f)\to C^{k+1}(M,f)\oplus \, C^{k-\ell+2}(M,f)
\end{align*}
with
\begin{align*}
d_C=\begin{pmatrix} d & \psi \\ 0 & (-1)^{\ell-1} d \end{pmatrix}\,,\qquad
\del_C= \begin{pmatrix} \del & c(\psi) \\ 0 & (-1)^{\ell-1}\del \end{pmatrix}\,.
\end{align*}
Such a chain map linking the two cone complexs,  which we will label by $\mP_C$,  must satisfy $\del_C \, \mP_C = \mP_C \, d_C$.  In fact, such a map exists and can be expressed in an upper-triangular matrix form.  
\begin{defn}
Let $\mP_C: \tC^\bullet(\psi)\to \tC^\bullet(c(\psi))$ be the upper-triangular matrix map
\[\mP_{C}=\begin{pmatrix} \mP& K \\ 0 & \mP\end{pmatrix}\]
where $K:\Omega^{k-\ell+1}(M) \to C^k(M, f)$, and acting on 
$\xi \in \Omega^{k-\ell+1}(M)$, $K$ is defined by 
\begin{align}\label{Kdef}
K\xi=(-1)^{\ell}\left(\mP\psi-\cpsi\mP\right)d^*G\xi\,+\,\del_{k, \perp}^{-1}\left(\left(\mP\psi-\cpsi\mP\right) \mathcal{H} \xi\,\right)\,,
\end{align}
expressed in terms of the Hodge decomposition with respect to the de Rham Laplacian $\Delta = dd^*+d^*d$:
\begin{align}\label{hodged}
\xi=(\mathcal{H}+\Delta G)\xi= \mathcal{H}\xi+dd^*G\xi+d^*dG\xi\,,
\end{align}
with $\mathcal{H}\xi$ denoting the harmonic component of $\xi$, and $G$, the Green's operator.
\end{defn}
Let us explain the notation $\del_{k, \perp}^{-1}$ in the second term of the definition for $K$ in \eqref{Kdef}.  Let $\gamma$ be a closed $(k-\ell+1)$-form.  Then from \eqref{c()coh}, we know that $\mP (\psi \wedge \gamma)$ and $c(\psi)\mP \gamma$  are cohomologous, and 
therefore, $\mP( \psi \wedge \gamma)-c(\psi)\mP \gamma=\del b$ for some $b \in C^{k}(M, f)$.  Note that $C^k(M, f)$ is an inner product space under $\langle q_{k_i}, q_{k_j}\rangle = \delta_{ij}$, so we have an orthogonal splitting, $C^{k}(M, f)=\ker\del_{k} \oplus (\ker\del_{k})^\perp$, and that $\del_k$ gives an isomorphism between $\left(C^{k}(M,f)/\ker\del_k\right)\cong (\ker\del_k)^\perp $ and  $\im\del_k\subset C^{k+1}(M,f)$. Thus, it follows from the finite-dimensional assumption on $C^k(M,f)$ and $C^{k+1}(M,f)$ that we can define a right inverse $\del_{k,\perp}^{-1}:\im \del_k \to (\ker\del_{k})^\perp\subset C^{k}(M, f)$, and $\del_{k, \perp}^{-1}(\mP( \psi \wedge \gamma)-c(\psi)\mP \gamma) \in C^{k}(M, f)$.  

With $\mP_C$ defined, we now show that it is a chain map.
 \begin{thm}
 $\mP_{C}:\tC^\bullet(\psi) \to \tC^\bullet(\cpsi)$ is a chain map.  In particular,
 \begin{align}\label{PCchain}
 \del_C\,\mP_{C}=\mP_{C}\,d_C\,.
 \end{align}
 \end{thm}
 \begin{proof}
The right and the left hand side of \eqref{PCchain} acting on $\eta\oplus \xi \in \Om^k(M)\oplus \Om^{k-\ell+1}(M)=\tC^k(\psi)$ give
 \begin{align*}
     \mP_{C}\,d_C&=\begin{pmatrix} \mP & K \\ 0 & \mP \end{pmatrix}\begin{pmatrix} d & \psi \\ 0 & (-1)^{\ell-1}d \end{pmatrix}=\begin{pmatrix} \mP d & \mP\psi+(-1)^{\ell-1}Kd \\ 0 & (-1)^{\ell-1}\mP d \end{pmatrix}, \\
     \del_C\,\mP_{C}&=\begin{pmatrix} \del & \cpsi \\ 0 & (-1)^{\ell-1}\del \end{pmatrix}\begin{pmatrix} \mP & K \\ 0 & \mP \end{pmatrix}=\begin{pmatrix} \del \mP & \del K +c(\psi)\mP\\ 0 & (-1)^{\ell-1}\del \mP \end{pmatrix}.
 \end{align*} 
 Since $\mP$ is a chain map \eqref{Pchain}, i.e. $d\,\mP=\mP\,\del$, the only entry we need to check comes from the off-diagonal one,  
 $$\mP\psi+(-1)^{\ell-1}Kd= \del K+c(\psi)\mP\,,$$ 
 or equivalently, we need to show that
 \begin{align}\label{Khom}
 \mP \psi - c(\psi)\mP= \del K+(-1)^{\ell}Kd\,, 
 \end{align}
 or $K$ is a graded chain homotopy. To compute $Kd\xi$, note first that 
 $\mathcal{H} d\xi=0\,$, $\forall \xi \in \Omega^{k-l+1}(M).$ Therefore, we find that
\[Kd\xi=(-1)^{\ell}\left(\mP\psi-\cpsi\mP\right)d^*Gd\xi=(-1)^{\ell}\left(\mP\psi -\cpsi\mP\right)d^*dG\xi\,,\]
having used the fact that $Gd=dG$.  Now, for the $\del K\xi$ term, we have 
\begin{align*}
    \del K\xi
    &=(-1)^\ell\del\left[\left(\mP \psi-\cpsi\mP\right) d^*G\xi \right]+\del\left[\del_{k, \perp}^{-1}\left(\left(\mP\psi-\cpsi\mP\right) \mathcal{H} \xi\, \right)\right] \\
    &=(-1)^{\ell}\left[(-1)^\ell\left(\mP \psi-\cpsi\mP\right) dd^*G\xi \right]+ \left(\mP \psi-\cpsi\mP\right) \mathcal{H}\xi \\
    &=\left(\mP\psi-\cpsi\mP\right) \left(dd^*G\xi+  \mathcal{H}\xi\right)
\end{align*} 
where in the second line, we have applied the graded commutative properties: $\del \mP=\mP d$ and $\del \cpsi = (-1)^\ell \cpsi \del$ for $\psi$ a $d$-closed $\ell$-form.  Altogether, we find for the right-hand side of \eqref{Khom}
\begin{align*}
\del K\xi+(-1)^{\ell}Kd\xi
&=\left(\mP \psi - \cpsi\mP\right) \left(dd^*G\xi+d^*dG\xi+\mathcal{H}\xi\right) \\
&=\left(\mP\psi -\cpsi\mP\right)\xi
\end{align*}
having used the Hodge decomposition formula of \eqref{hodged}.  
Thus, we have proved that $K$ is a graded chain homotopy, and therefore, $\mP_{C}\,d_C=\del_C\, \mP_{C}$. 
\end{proof}

\subsection{Isomorphism of cohomologies via Five Lemma}
A mapping cone cochain complex can be described by a short exact sequence of chain maps.  For the differential forms case, we have
\begin{equation}\label{Cses}
\begin{tikzcd}
0 \arrow[r]  &
(\Omega^k(M),d) \arrow[r,"\iota_{dR}"]  &
(\tC^k(\psi), d_C) \arrow[r,"\pi_{dR}"]  & (\Omega^{k-\ell+1}(M), (-1)^{\ell-1} d) \arrow[r] & 0
\end{tikzcd}
\end{equation}
where $\iota_{dR}$ is the inclusion into the first component $\iota_{dR}(\eta)=\begin{pmatrix} \eta \\ 0 \end{pmatrix}$ and $\pi_{dR}$ is the projection of the second component $\pi_{dR}\begin{pmatrix} \eta \\ \xi \end{pmatrix}=\xi$.  It is easy to check that these maps are chain maps:
\begin{align*}
\iota_{dR} d\eta =\begin{pmatrix} d\eta \\ 0 \end{pmatrix}=d_C\iota_{dR}\eta
\end{align*} 
and
\begin{align*}
\pi_{dR}\,d_C\begin{pmatrix} \eta \\ \xi \end{pmatrix} =\pi_{dR}\begin{pmatrix}d\eta + \psi \wedge \xi \\ (-1)^{\ell-1}d\xi \end{pmatrix}=(-1)^{\ell-1}d\xi=(-1)^{\ell-1}d\left\{\pi_{dR}\begin{pmatrix} \eta \\ \xi \end{pmatrix}\right\}\,.
\end{align*}
The short exact sequence \eqref{Cses} implies the following long exact sequence for the cohomology of Cone$(\psi)$ 
\begin{equation}\label{CHes}
\begin{tikzcd} 
\ldots\arrow[r]&H^{k-\ell}_{dR}(M) \arrow[r, "\text{[}\psi\text{]}"] & H^{k}_{dR}(M) \arrow[r, "\text{[}\iota_{dR}\text{]}"] & H^{k}(\tC(\psi)) \arrow[r, "\text{[}\pi_{dR}\text{]}"] & H^{k-\ell+1}_{dR}(M)
\arrow[r]&\dots
\end{tikzcd}
\end{equation}
Analogously, for $\tC(\cpsi)$, we also have the short exact sequence of chain maps
\begin{equation}\label{CMses}
\begin{tikzcd}
0  \arrow[r] & (C^k(M,f), \del) \arrow[r, "\iota_{C(f)}"] & (\tC^k(\cpsi), \del_C) \arrow[r, "\pi_{C(f)}"] & (C^{k-\ell+1}(M,f), (-1)^{\ell-1}\del) \arrow[r]& 0 
\end{tikzcd}
\end{equation}
and the long exact sequence of cohomology
\begin{equation}\label{CMHes}
\begin{tikzcd} 
\mspace{-5mu}\ldots\mspace{-3mu} \arrow[r]&\mspace{-3mu} H^{k-\ell}_{C(f)}(M) \mspace{-1mu}\arrow[r, "\text{[}\cpsi\text{]}"] & H^{k}_{C(f)}(M) \arrow[r, "\text{[}\iota_{C(f)}\text{]}"] & H^{k}(\tC(\cpsi)) \arrow[r, "\text{[}\pi_{C(f)}\text{]}"] & H^{k-\ell+1}_{C(f)}(M) \mspace{-3mu}
\arrow[r]& \mspace{-3mu}\dots
\end{tikzcd}
\end{equation}

The two short exact sequences, \eqref{Cses} and \eqref{CMses}, fit into a commutative diagram. 
\begin{equation}\label{Dses}
\begin{tikzcd}
0 \arrow[r]  &
(\Omega^k(M),d) \arrow[r,"\iota_{dR}"] \arrow[d,swap,"\mP"] &
(\tC^k(\psi)), d_C) \arrow[r,"\pi_{dR}"] \arrow[d,swap,"\mP_{C}"] & (\Omega^{k-\ell+1}(M), (-1)^{\ell-1}d) \arrow[d,swap,"\mP"] \arrow[r] & 0
\\
0 \arrow[r] & (C^k(M,f), \del) \arrow[r, "\iota_{C(f)}"] & (\tC^k(\cpsi), \del_C) \arrow[r, "\pi_{C(f)}"] & (C^{k-\ell+1}(M,f), (-1)^{\ell-1}\del) \arrow[r]& 0 
\end{tikzcd}
\end{equation}
The commutativity of the above diagram can be checked as follows: 
\[\iota_{C(f)}(\mP(\eta))=\begin{pmatrix} \mP\eta \\ 0 \end{pmatrix}=\begin{pmatrix} \mP & K \\ 0 & \mP \end{pmatrix}\begin{pmatrix} \eta \\ 0 \end{pmatrix}=\mP_{C}(\iota_{dR}(\eta))\,,\]
\[\pi_{C(f)}\left(\mP_C\begin{pmatrix} \eta \\ \xi \end{pmatrix}\right)=\pi_{C(f)}\begin{pmatrix} \mP\eta+K\xi \\ \mP\xi \end{pmatrix}=\mP\xi=\mP\left(\pi_{dR}\begin{pmatrix} \eta \\ \xi \end{pmatrix}\right).\] 
The short exact commutative diagram \eqref{Dses} gives a long commutative diagram of cohomologies:
\begin{equation}\label{DHes}
\begin{tikzcd} 
H^{k-\ell}_{dR}(M) \arrow[r, "\text{[}\psi\text{]}"] \arrow[d,swap,"\text{[}\mP\text{]}"] & H^{k}_{dR}(M) \arrow[d,swap,"\text{[}\mP\text{]}"] \arrow[r, "\text{[}\iota_{dR}\text{]}"] & H^{k}(\tC(\psi)) \arrow[d,swap,"\text{[}\mP_{C}\text{]}"] \arrow[r, "\text{[}\pi_{dR}\text{]}"] & H^{k-\ell+1}_{dR}(M) \arrow[d,swap,"\text{[}\mP\text{]}"] \arrow[r, "\text{[}\psi\text{]}"] & H^{k+1}_{dR}(M) \arrow[d,swap,"\text{[}\mP\text{]}"] \\ 
H^{k-\ell}_{C(f)}(M) \arrow[r, "\text{[}\cpsi\text{]}"] & H^{k}_{C(f)}(M) \arrow[r, "\text{[}\iota_{C(f)}\text{]}"] & H^{k}(\tC(\cpsi)) \arrow[r, "\text{[}\pi_{C(f)}\text{]}"] & H^{k-\ell+1}_{C(f)}(M) \arrow[r, "\text{[}\cpsi\text{]}"] & H^{k+1}_{C(f)}(M)
\end{tikzcd}
\end{equation}
We can check that each square commutes.  The outer squares commute since  $\mP (\psi \wedge \xi)$ and $c(\psi)\mP\xi$ are cohomologous when both $\xi$ and $\psi$ are $d$-closed, by \eqref{c()coh}, as was shown by Austin-Braam in \cite{AB}*{Section  3.5}. 
The middle two squares commute following from the commutativity of the chain maps in \eqref{Dses}.  Furthermore, the vertical map $[\mP]$ is an isomorphism \eqref{Pisom} as shown by Bismut-Zhang and Laudenbach \cite{BZ}*{Theorem 2.9} (see also  \cite{Zhang}*{Theorem 6.4}). 

We can now apply the Five Lemma to \eqref{DHes} which implies that the middle vertical map $[\mP_{C}]$ is also an isomorphism on cohomology, and thus we have  proved Theorem \ref{MIso}.
\begin{thm}
 $\mP_{C}:(\tC^\bullet(\psi), d_C)\to (\tC^\bullet(\cpsi), \del_C)$ is a $\mathbb{Z}$ graded quasi-isomorphism.
\end{thm} 



\section{Cone Morse inequalities}

Having established the quasi-isomorphism between the complexes, Cone$(\psi)$ and Cone$(c(\psi))$, we will proceed now to prove Theorem \ref{CMineq}, which gives the Morse-type inequalities for the Cone$(\psi)$ complex analogous to those \eqref{sm2}-\eqref{wm1} for the de Rham complex.  

For a closed, oriented manifold $M$ and a $d$-closed form $\psi\in \Omega^\ell(M)$, let us denote by $b_k^\psi = \dim H^k(\text{Cone}(\psi))$.  From \eqref{hckk}, we know that 
\begin{align}\label{ConeH}
H^{k}(\tC(\psi))\cong \coker\left([\psi]\!: H_{dR}^{k-\ell}\to H_{dR}^k\right)\,\oplus\,\, \ker\left([\psi]\!: H_{dR}^{k-\ell+1}\to H_{dR}^{k+1}\right)
\end{align}
which implies
\begin{align}\label{bpsik}
    b^\psi_k &=\dim \left[\coker \left([\psi]:H_{dR}^{k-\ell} \to H_{dR}^{k}\right)\right]+\dim \left[\ker \left([\psi]:H_{dR}^{k-\ell+1} \to H_{dR}^{k+1}\right)\right]\nonumber\\
    &=b_k-\txw_{k-\ell}+b_{k-\ell+1}-\txw_{k-\ell+1}  
\end{align}
where $b_k=\dim H_{dR}^k(M)$ and 
\begin{align}\label{rkdef}
\txw_k=\rk\left( [\psi]:H_{dR}^{k}(M) \to H_{dR}^{k+\ell}(M)\right)\,.
\end{align}

We would like to bound $b^\psi_k$ by means of the Morse function and properties of the cone Morse complex $\tC(\cpsi)$.  That $H(\tC(\psi))$ as expressed above is related to the cokernel and kernel of the $\psi$ map is suggestive that we should look for an analogous relationship between $H(\tC(\cpsi))$ with the cokernel and kernel of the $\cpsi$ map.  Indeed, such a relationship exists for any cone complex.  (See  \cite{Weibel} or Appendix \ref{AppB} for a review.)  For the Morse complex $(C^\bullet(M, f), \partial)$, we will make use of two subcomplexes, the kernel and cokernel complex, associated to the map $\cpsi$:  
\begin{itemize}
\item The kernel complex of $\cpsi$, $(\ker \cpsi, \del)$, is the complex consisting of $\ker^j\! \cpsi =\{b \in C^j(M,f) \,|\,\cpsi b =0\}$, with differential $\del$.
\item The cokernel complex of $\cpsi$, $(\coker \cpsi, \del^\pi)$, is the complex $\coker^j\!\cpsi=\{[a] \in C^j/\im \cpsi \}$ with differential $\del^\pi [a]=[\del a] \in C/\im \cpsi$.
\end{itemize}
The cohomologies of these two subcomplexes together with $H(\tC(\cpsi))$ forms a long exact sequence \eqref{leskca}
\begin{align}\label{leskc}
\dots \xrightarrow{~~~} H^{k-\ell+1}(\ker\cpsi) \xrightarrow[]{h^{k-\ell+1}_{ker}} H^{k}(\text{Cone}(\cpsi))\xrightarrow[]{h^k_{Cone}} H^{k}(\coker\cpsi) \xrightarrow[]{h^k_{coker}} \dots 
\end{align} 
The precise definitions of the maps in the long exact sequence will not be needed in our discussion here.  From \eqref{leskc}, we can immediately obtain the following weak cone Morse inequality.
\begin{thm}[Weak Cone Morse Inequalities]
On a closed manifold M with $\psi\in \Om^\ell(M)$ a $d$-closed form, let $b^\psi_k = \dim H^k(\tC(\psi))$ and $m_k$ the number of index $k$ critical points of a Morse function on $M$.  Then, we have for  $k = 0, 1, \ldots, (\dim M + \ell -1)$,
\begin{align}\label{wcMIneq}
b^\psi_k \leq m_k - v_{k-\ell}+m_{k-\ell+1}-v_{k-\ell+1}\,, 
\end{align} 
where 
\begin{align}\label{vkdef}
v_k=\rk\left(\cpsi:C^{k}(M,f) \to C^{k+\ell}(M,f)\right)\,.
\end{align}
\end{thm}
\begin{proof}
From \eqref{leskc}, we have
\begin{align*}
b_k^\psi 
& \leq   \dim H^{k}(\coker\cpsi) + \dim H^{k-\ell+1}(\ker\cpsi)\\
& \leq   \dim (\coker^k\!\cpsi)  + \dim (\ker^{k-\ell+1}\!\cpsi) \\
& =  m_k - v_{k-\ell} + m_{k-\ell+1}-v_{k-\ell+1} \,.
\end{align*}
\end{proof}
In general, the number $v_k=\rk\cpsi |_{C^k(M,f)}$ is not equal to $r_k=\rk \, [\psi]|_{H_{dR}^k(M)}$ as defined in \eqref{rkdef}.  However, we have the following relations.
\begin{lem}\label{rvrel}
 Let $\txw_k=\rk\, [\psi]|_{H_{dR}^k(M)}$ and $v_k =\rk \cpsi |_{C^k(M,f)}$ as defined in \eqref{rkdef} and \eqref{vkdef}, respectively. 
Then for $k = 0, 1, \ldots, (\dim M + \ell -1)$,
\begin{itemize}
\item[(a)] $r_k=\rk \left([\cpsi]: H^k_{C(f)}(M) \to H^{k+\ell}_{C(f)}(M)\right)$;
\item[(b)] $\txw_k \leq v_k$; 
\item[(c)] $b_k-v_{k-\ell}+b_{k-\ell+1}-v_{k-\ell+1} \leq b_k^{\psi} \leq m_k-\txw_{k-\ell}+m_{k-\ell+1}-\txw_{k-\ell+1}$;
\item[(d)] $(v_{k-\ell}-\txw_{k-\ell})+(v_{k-\ell+1}-\txw_{k-\ell+1})\leq (m_k-b_k)+(m_{k-\ell+1}-b_{k-\ell+1})\,$.
\end{itemize}
\end{lem}
\begin{proof}
Property (a) follows from \eqref{Pisom}-\eqref{c()coh} which implies that $\text{rank } [\psi]= \text{rank } [\cpsi]$.  For (b), if $\{[\cpsi a_1], ... [\cpsi a_{r_k}]\}$ gives a basis for $\im\, [\cpsi]\subset H^{k+\ell}_{C(f)}(M)$, then $\{\cpsi a_1, ... , \cpsi a_{r_k}\}$ must constitute a linearly independent set of elements in $\im \cpsi\subset C^{k+\ell}(M, f)$, and therefore,
\[r_k \leq \dim \left(\im \cpsi \cap C^{k+\ell}(M, f)\right) = v_k\,.\]
Applying property (b) to \eqref{bpsik} and \eqref{wcMIneq} results in property (c). 
Lastly, property (d) follows from combining \eqref{bpsik} and \eqref{wcMIneq}.    
\end{proof}
From the standard Morse inequality, $b_k \leq m_k$, and Lemma \ref{rvrel}(b), $r_k \leq v_k$,  we see that the relation in Lemma \ref{rvrel}(d)
\begin{align}\label{posrel}
(v_{k-\ell}-\txw_{k-\ell})+(v_{k-\ell+1}-\txw_{k-\ell+1})\leq (m_k-b_k)+(m_{k-\ell+1}-b_{k-\ell+1})
\end{align}
consist of sums of two non-negative terms on both sides.  In particular, if the Morse function $f$ is perfect, i.e. $m_k=b_k$, 
then \eqref{posrel}  implies the following result.
\begin{cor}\label{wperf}
 If $f$ is a perfect Morse function, then $v_k = r_k$ and 
\begin{align}
b^\psi_k&= (m_{k}-v_{k-\ell}+m_{k-\ell+1}-v_{k-\ell+1})  
\end{align}
for $k=0, 1, \ldots, \dim M + \ell -1$.
\end{cor}
Hence, for a perfect Morse function, the weak cone Morse inequality becomes an equality.  And this is as expected since a perfect Morse function implies for the Morse complex that $\dim H_{C(f)}^k(M) =\dim C^k(M,f)$, and therefore, $[\cpsi]|_{H_{C(f)}}$ and $\cpsi|_{C^k(M,f)}$ are the same map.  
Clearly, Equation  \eqref{posrel} constrains the deviations of the $v_k$'s from the $r_k$'s by the deviations of the $m_k$'s from the $b_k$'s.  

We now proceed to prove the strong cone Morse inequalities. 
\begin{thm}[Strong Cone Morse Inequalities]\label{SMTI}
On a closed manifold M with $\psi\in \Om^\ell(M)$ a $d$-closed form, let $b^\psi_k = \dim H^k(\tC(\psi))$, $m_k$ be the number of index $k$ critical points of a Morse function on $M$, and $v_k =\rk \cpsi |_{C^k(M,f)}$.  Then, we have for  $j = 0, 1, \ldots, (\dim M + \ell -1)$,
\begin{align}\label{scMIneq}
\sum_{k=0}^j(-1)^{j-k} b_k^\psi \leq \sum_{k=0}^j(-1)^{j-k}(m_k-v_{k-\ell}+m_{k-\ell+1}-v_{k-\ell+1}).
\end{align}
\end{thm}
\begin{proof}
The $j=0$ case is just the weak inequality of \eqref{wcMIneq}.  So we can assume $j\geq 1$.  We note first that \eqref{leskc} implies
\begin{align}
    0 \to \im h^{k-\ell+1}_{ker} \xrightarrow[]{} H^{k}(\tC(\cpsi)) \xrightarrow[]{} \im h^k_{Cone} \to 0 \label{ses1}\\
    0 \to \im h^{k}_{Cone} \xrightarrow[]{} H^{k}(\coker\cpsi) \xrightarrow[]{} \im h^k_{coker} \to 0 \label{ses2}\\
    0 \to \im h^{k-1}_{coker} \xrightarrow[]{} H^{k-\ell+1}(\ker\cpsi)  \xrightarrow[]{} \im h^{k-\ell+1}_{ker} \to 0 \label{ses3}
\end{align}
By Theorem \ref{MIso}, $b^\psi_k = \dim H^k(\text{Cone}(\cpsi))$.  Thus, we can use  \eqref{ses1} to write
\begin{align}
\sum_{k=0}^j(-1)^{j-k} b^\psi_k&= \sum_{k=0}^j(-1)^{j-k} \left[ \dim (\im h^{k}_{Cone}) +\dim (\im h^{k-\ell+1}_{ker})\right] \nonumber\\
&=\dim (\im h^{j}_{Cone})- \sum_{k=0}^j(-1)^{j-k} \left[  \dim (\im h^{k-1}_{Cone}) - \dim (\im h^{k-\ell+1}_{ker}) \right] \nonumber\\
&=\mspace{-1mu} \dim (\im h^{j}_{Cone})\mspace{-2mu}-\mspace{-3mu} \sum_{k=0}^j(-1)^{j-k}\mspace{-1mu}\left[ \mspace{-1mu} \dim\mspace{-1mu} H^{k-1}(\coker\cpsi) - \dim\mspace{-1mu}H^{k-\ell+1}(\ker\cpsi)\mspace{-1mu}\right] \nonumber\\
&\leq \sum_{k=0}^j(-1)^{j-k} \left[\dim H^{k}(\coker\cpsi)+\dim H^{k-\ell+1}(\ker\cpsi)\right] \label{smineq}
\end{align}
where in the third line, we used \eqref{ses2}-\eqref{ses3}, and in the fourth line \eqref{ses2} again.  Now, because $M$ is assumed to be a closed manifold, both the $\ker\cpsi$ and the $\coker \cpsi$  complex are finitely generated.  In general, for any finitely-generated cochain complex 
$0 \xrightarrow[]{} C^0 \xrightarrow[]{\partial} C^1\xrightarrow[]{\partial} C^2 \xrightarrow[]{\partial} \dots\,,$ 
the dimensions of the associated cohomologies and that of the cochains satisfy the following inequality:
\begin{align*}
\sum_{k=0}^j(-1)^{j-k} \dim H^k(C)  \leq \sum_{k=0}^j(-1)^{j-k}\dim C^k\,.
\end{align*}  
Applying this relation to  \eqref{smineq} results in
\begin{align*}
\sum_{k=0}^j(-1)^{j-k} b^\psi_k 
&\leq \sum_{k=0}^j(-1)^{j-k} \left[\dim(\coker^k\!\cpsi)+ \dim(\ker^{k-\ell+1}\!\cpsi)\right]\\
&= \sum_{k=0}^j(-1)^{j-k}(m_k-v_{k-\ell}+ m_{k-\ell+1}-v_{k-\ell+1})\,.
\end{align*} 
Thus, we obtain the strong cone Morse inequality for the cone complex.
\end{proof}
In the case where $f$ is a perfect Morse function, Corollary \ref{wperf} immediately implies the following.
\begin{cor}\label{perf} 
If $f$ is a perfect Morse function, then the strong cone Morse inequalities become equalities:
\begin{align}
 \sum_{k=0}^j(-1)^{j-k}b^\psi_k &=  \sum_{k=0}^j(-1)^{j-k}(m_{k}-v_{k-\ell}+m_{k-\ell+1}-v_{k-\ell+1})\,. 
\end{align}
\end{cor}
More generally, when the Morse function is not perfect, Theorem \ref{SMTI} implies an analogous strong-version of the inequalities in Lemma \ref{rvrel}(d).

\begin{cor}
For $j=0, \ldots, \dim M + \ell-1$,  
\begin{align*}
    \sum_{k=0}^{j} (-1)^{j-k}((v_{k-\ell}-\txw_{k-\ell})+&(v_{k-\ell+1}-\txw_{k-\ell+1}))\\
    \leq &\sum_{k=0}^{j} (-1)^{j-k}((m_k-b_k)+(m_{k-\ell+1}-b_{k-\ell+1})).
\end{align*}
\begin{proof}
By \eqref{bpsik} and Theorem \ref{SMTI}, we have
\begin{align*}
    \displaystyle\sum_{k=0}^{j} (-1)^{j-k}(b_k-\txw_{k-\ell}+b_{k-\ell+1}-\txw_{k-\ell+1})&=\displaystyle\sum_{k=0}^{j} (-1)^{j-k}b^\psi_k\\
    &\leq \displaystyle\sum_{k=0}^{j} (-1)^{j-k}(m_k- v_{k-\ell}+m_{k-\ell+1}-v_{k-\ell+1})\,.
\end{align*}
The relation of the Corollary is then obtained by moving the $b_k$'s to the right-hand-side and the $v_k$'s to the left-hand-side.
\end{proof}
\end{cor}
 
In the special case where $\psi$ is a closed two-form, i.e. $\ell=2$, Theorem \ref{perf} results in an interesting relation.  
\begin{cor}\label{l2e}
For a closed two-form $\psi$, we have the bounds for $k=0, \ldots, \dim M -1$, 
\begin{align*}
0 \leq v_{k}-r_{k} \leq m_{k+1}-b_{k+1}\,.
\end{align*}
\end{cor}
\begin{proof}
When $\ell=2$,  equation \eqref{bpsik} implies for $j\geq 1$
\begin{align}\label{l2e1}
\displaystyle\sum_{k=0}^{j} (-1)^{j-k}b^\psi_k = \displaystyle\sum_{k=0}^{j} (-1)^{j-k}(b_k+r_{k-2}- b_{k-1}-r_{k-1}) = b_j - r_{j-1}\,,
\end{align}
and similarly, Theorem \ref{SMTI} implies for $j\geq 1$
\begin{align}\label{l2e2}
\displaystyle\sum_{k=0}^{j} (-1)^{j-k}b^\psi_k \leq\displaystyle\sum_{k=0}^{j} (-1)^{j-k}(m_k- v_{k-2}+m_{k-1}-v_{k-1})=m_j-v_{j-1}\,.
\end{align}
Combining \eqref{l2e1}-\eqref{l2e2} gives the relation $v_{j-1}-r_{j-1} \leq m_{j}-b_{j}$.
\end{proof}
Corollary \ref{l2e} interestingly shows that when $\ell=2$, the rank of the $\cpsi$ map on $C^k(M,f)$ is constrained, not just by $m_k$, the number of critical points of index $k$, as would be expected, but also by $m_{k+1}$, relative to $b_{k+1}$.  Corollary \ref{l2e} also gives a bound for the difference between $m_k$ and $b_k$.  The bound becomes especially simple in the case when $\psi$ is an exact form.  For an exact $\psi$ implies $r_k=0$ and we thus obtain the relation in Corollary \ref{bmv} 
\begin{align}\label{bmvIneq1}
b_k \leq m_k - v_{k-1}\,,\qquad k=1, \ldots, \dim M\,.  
\end{align}


\section{Example of the two-sphere}

To illustrate the properties described in the previous section, we here consider the cone Morse cohomology and its Morse-type inequalities in the context of the two-sphere.  We will show explicitly that the cone cohomology $H(\tC(\psi))\cong H(\tC(\com))$ can vary with $[\psi]$, and also present examples where the cone Morse bounds vary under change of the closed-form $\psi$, the metric $g$, and the Morse function $f$.

Let $M=S^2=\{(x,y,z)\in\mathbb{R}^3\,|\, x^2+y^2+z^2=1\}$ be the unit sphere in $\mathbb{R}^3$.  In our discussion, we will interchangeably use both spherical coordinates $(\phi, \theta)$ and Euclidean coordinates $(x, y, z) = (\sin\phi \cos \theta, \sin\phi \sin \theta, \cos\phi)$ to describe functions and forms on $S^2$.  We will denote by $\om_0 := \sin\phi \,d\phi \w d\theta$ the standard volume form with volume $\int_{S^2} \om_0 = 4\pi\,$.    Throughout, we will let $\psi\in \Om^2(S^2)$.  Note that in two dimensions, all two forms are trivially $d$-closed.  
  
The dimension of the cone cohomology, $b_k^\psi = \dim H^k(\tC(\psi))$, can be quickly calculated using \eqref{bpsik}.  We find
\begin{align}\label{bStg}
b^\psi_k = 
\begin{cases}
1 &  k =0, 3 \\
1-r_0  &  k =1, 2
\end{cases}
\end{align}
where $r_0=\rk [\psi] |_{H^0_{dR}}$.  Let us note that $r_0=1$ whenever $[\psi]\in H^2_{dR}(S^2)$ is a non-trivial class, or equivalently, $\int_{S^2} \psi \neq 0$. 

\begin{ex}\label{ex1} {\bf [Change in $b^\psi_k=\dim H^k(\tC(\psi))$ as $[\psi]$ varies.]} 
Consider the one-parameter family 
\begin{align}\label{psis}
\psi_s=(y+s)\,\om_0\,, \qquad\qquad s\in [-1, 1]\,.
\end{align}
Note that $\psi_s$ is not a symplectic form as it vanishes along $y=-s$.  In fact,  the cohomology class $[\psi_s]= s [\om_0]$, and therefore, $r_0=1$ for $s \neq 0$ and $r_0=0$ for $s=0$.  It follows from \eqref{bStg} that at the special value of $s=0$, $b^{\psi_s}_k$ increases by one for $k=1,2$.
\end{ex}
    
\medskip    
    
We now turn to the cone Morse inequalities.  To do so, we need to introduce a Morse function on $S^2$.  The generic Morse function is not perfect.  An example of a non-perfect Morse function which we will use for the remainder of this section is    
\begin{align}\label{S2Morsef}
f=x^2+2y^2 + 3 z^2
\end{align}
restricted to $S^2$.  This Morse function has six critical points which can be easily seen by expressing $f$ in terms of only two variables applying the unit circle constraint (see Figure \ref{S26pts}):
\begin{figure}[t]
    \centering
\includegraphics[scale=.3]{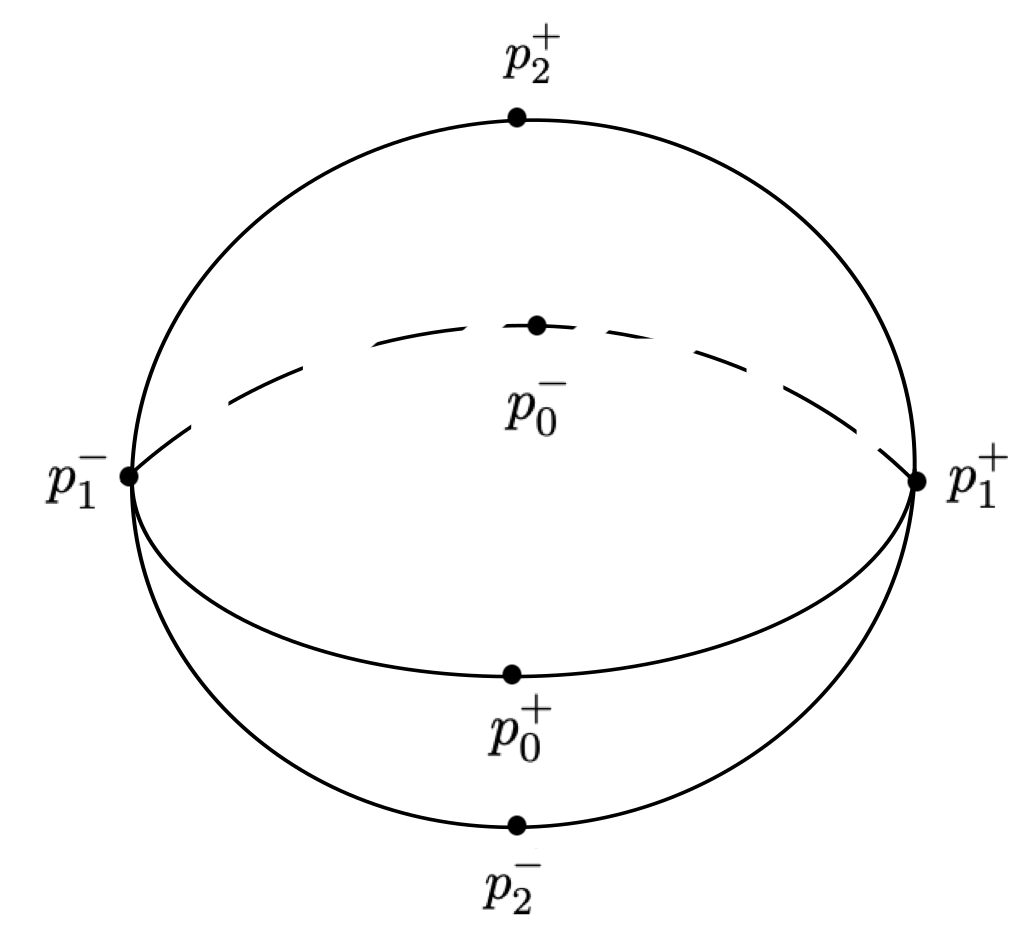}
\caption{Two sphere in $\mathbb{R}^3$ with six critical points of the Morse function $f=x^2+2y^2+3z^2$.}
    \label{S26pts}
\end{figure} 
\begin{alignat*}{3}
f&=1+ y^2 + 2z^2 &\qquad  &(y^2 + z^2 \leq 1)\qquad\qquad&  \text{index~0~critical~points:~~}& p^\pm_0 = (\pm 1 , 0 , 0 )\,; \\
f&=2-x^2+z^2 &\qquad &(x^2 + z^2 \leq 1)\qquad\qquad& \text{index~1~critical~points:~~}
&p^\pm_1 = (0 , \pm 1 , 0 )\,;\\
f&=3-2x^2-y^2 &\qquad &(x^2 + y^2 \leq 1)\qquad\qquad&  \text{index~2~critical~points:~~} &p^\pm_2 = (0 , 0, \pm 1 )\;.
\end{alignat*}
Clearly, $m_k \neq b_k(S^2)$.  Below, we will write a generic element of $C^k(S^2, f)$ for $k=0, 1, 2$, each generated by two critical points, by the linear combination $a_k^+ p_k^+ + a_k^- p_k^-$  with the coefficients $a_k^+, a_k^- \in \mathbb{R}\,$.

As for the Riemannian metric, we will use as default $g_0 = d\phi^2 + \sin^2 \phi \,\, d\theta^2$, the standard sphere metric induced from the standard Euclidean metric on $\mathbb{R}^3$.  The pair $(f, g_0)$ determines the moduli space of gradient flow lines which goes into the calculation of $\com: C^k(S^2, f) \to C^{k+2}(S^2, f)$.  Only for $k=0$ is the $\com$ map non-trivial and this corresponds to integrating $\psi$ over the moduli space  $\overline{\CM(p_2^{\pm}, p_0^{\pm})}$ which are the four quarter spheres determined by ($x\leq 0$ or $x\geq 0$) and $(z\leq0$ or $z\geq 0$).

The weak cone Morse inequality in \eqref{wcMIneq} gives for $(S^2, f, g_0)$ the bound
\begin{align}\label{wmStg}
b^\psi_k \leq 
\begin{cases}
2 &  k =0, 3 \\
4-v_0  &  k =1,2
\end{cases}
\end{align}
where $v_0 = \rk c(\psi)|_{C^0(S^2, f)}$.  In the following three examples below, we will demonstrate the dependence of the above cone Morse bound, and specifically the dependence of $v_0$, on $\psi$, the metric, and the Morse function.
 
\begin{ex}{\bf [Change in the cone Morse bound as $[\psi]$ varies.]}
We take $\psi$ to be again the one-parameter family $\psi_s$ in \eqref{psis} of Example \ref{ex1}. Let us calculate $v_0 = \rk c(\psi_s)|_{C^0(S^2, f)}$.  The operator $c(\psi_s)$ acting on $a_0^+ p_0^+ + a_0^- p_0^- \in C^0(S^2, f)\,$ generated by the two index zero critical points $(p^+_0, p^-_0)$ can be found by integrating $\psi_s$ over the quarter spheres, and has the following matrix form:
\[\begin{pmatrix}
    a_2^+ \\ a_2^-
\end{pmatrix}=c(\psi_s) \begin{pmatrix}
    a_0^+ \\ a_0^-
\end{pmatrix}= 
\pi\begin{pmatrix} 
s\,   & s\, \\ s\,  & s\,
\end{pmatrix}
\begin{pmatrix}
    a_0^+ \\ a_0^-
\end{pmatrix}\,,\]
mapping into $a_2^+ p_2^+ + a_2^- p_2^- \in C^2(S^2, f)\,$.   
For the rank, we find that $v_0=1$ for $s\neq 0$, and $v_0=0$ for $s=0$.  Hence, by \eqref{wmStg}, the weak cone Morse bound for $b^{\psi_s}_k$ for $k=1,2$ increases by one at $s=0$.  This coincides with the increase in $b^{\psi_s}_k$ at $s=0$ as calculated in Example \ref{ex1}.
\end{ex}

\begin{ex}\label{ex3}{\bf [Change in the cone Morse bound as $\psi$ varies within a fixed de Rham class.]}  The weak cone Morse bound can also vary within the same de Rham cohomology class $[\psi]$. Consider the following one-parameter family of $\psi$:
\begin{align}\label{psit}
\psi_t=(1+tx+tz)\,\om_0\,, \qquad -\dfrac{1}{2}<t<\dfrac{1}{2}\,.
\end{align}
Note that $[\psi_t]=[\om_0]\in H^2_{dR}(S^2)$ for all $t\in (-\frac{1}{2}, \frac{1}{2})$, and so $b^{\psi_t}_k$ does not vary.  However, for the weak cone Morse bound, the $c(\psi_t)$ map takes the form
\[\begin{pmatrix}
    a_2^+ \\ a_2^-
\end{pmatrix}=c(\psi_t) \begin{pmatrix}
    a_0^+ \\ a_0^-
\end{pmatrix}=\pi \begin{pmatrix} 1+t & 1 \\ 1 & 1-t \end{pmatrix}\begin{pmatrix}
    a_0^+ \\ a_0^-
\end{pmatrix}\]
and has rank $v_0=2$ for $t\neq 0$, and $v_0=1$ for $t=0$.  This gives the bound for $k=1,2$
\begin{align*}
b^{\psi_t}_k \leq 
\begin{cases}
3  &  t=0 \\
2  &  t\neq 0
\end{cases}
\end{align*}
even though $b^{\psi_t}_k$ remains constant.  This example demonstrates that the $v_k$'s generally depend on $\psi$ as a closed form, in contrast with the $b^\psi_k${}'s, which only vary with the cohomology class $[\psi]$. 
\end{ex}

\begin{rmk}
Notice that the one-parameter family of closed two-forms $\psi_t$ in \eqref{psit} are all non-degenerate, and hence, symplectic.  Being in the same cohomology class, Moser's theorem implies the existence of  a one-parameter family of symplectomorphism $\varphi_t: S^2\to S^2$ such that $\varphi^*_t\om_t = \om_0$.  We can use this symplectomorphism to pull back $(S^2, \om_t, f, g_0)$ to $(S^2, \om_0, \varphi^*_t f, \varphi^*_t g_0)$.  As symplectomorphisms leave unchanged $m_k$'s and $v_k$'s, we can reinterpret the above example as varying the Morse-Smale pair $(\varphi^*_t f, \varphi^*_t g_0)$ while keeping fixed the closed-form $\psi=\om_0$ on $S^2$. It thus also represents an example where the cone Morse bound $b^\psi_k$ changes when the Morse-Smale pair $(f, g)$ is varied.
\end{rmk}
\begin{ex}{\bf [Change in the cone Morse bound as the Riemannian metric $g$ varies.]}
\begin{figure}[t]
    \centering
\includegraphics[scale=.3]{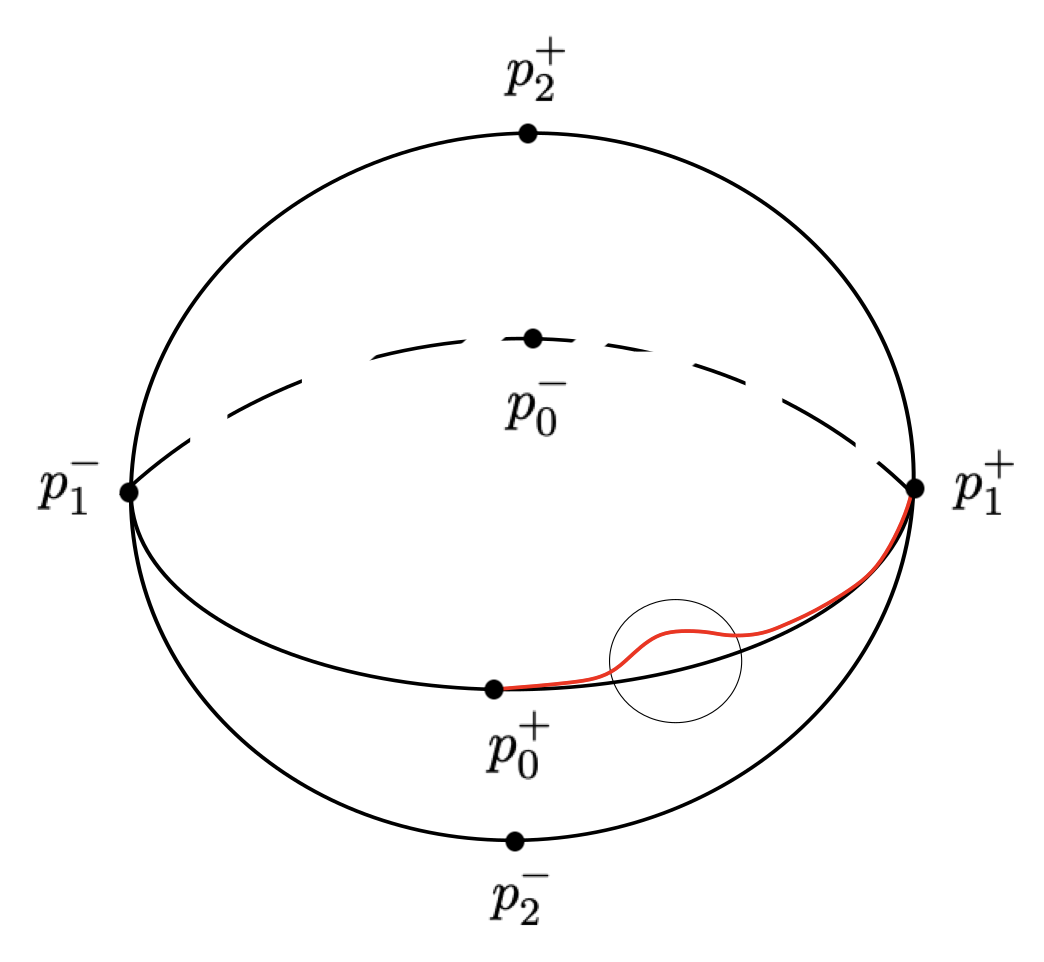} 
\caption{Two sphere with modified metric deforming the flow line from $p_1^+$ to $p_0^+$.}
    \label{S2MetMod}
\end{figure}
We demonstrate here that the bound for $b^\psi_k$ can jump just by varying the metric.  Let $\psi=\om_0$ the standard area form on $S^2$. For the standard round metric, $g_0$, the moduli space $\CM(p^\pm_2, p^\pm_0)$ of flow lines are just the four quarters of the sphere.  Suppose we modify this round metric within a small neighborhood of a point that is on the flow line between $p^+_1$ and $p^+_0$ (the small circle on $S^2$ in Figure \ref{S2MetMod}). In doing so, we can change the gradient flow lines, so the new boundary is the red line above. As such that we remove an area of $\epsilon$ that is between the red flow line and the original black line. Thus we subtract $\epsilon$ from $\CM(p_2^+, p_0^+)$ and add it to $\CM(p_2^-, p_0^+)$ 
\[\begin{pmatrix}
    a_2^+ \\ a_2^-
\end{pmatrix}=c(\om_0) \begin{pmatrix}
    a_0^+ \\ a_0^-
\end{pmatrix}=\pi \begin{pmatrix} 1-\epsilon & 1 \\ 1+\epsilon & 1 \end{pmatrix}\begin{pmatrix}
    a_0^+ \\ a_0^-
\end{pmatrix}\]
In this case, the rank of $c(\om_0)|_{C^0(S^2,f)}$ jumps to $v_0=2$ when $\epsilon\neq 0$. By \eqref{wmStg}, this correspondingly decreases the bound on $b^{\om_0}_k$ for $k=1,2,$ by one, and hence, gives an explicit example where the bound varies with the metric.
\end{ex}
\begin{rmk}
The cone cohomology dimension $b^\psi_k$ depends only on the cohomology class $[\psi]$.  We have seen above how the cone Morse inequalities, can explicitly depend on the Morse function $f$, the metric $g$ and even the representative form $\psi$ in $[\psi]$.  We can improve the bound by varying $g$ and $\psi$ within the class $[\psi]$ to maximize $v_k=\rk c(\psi)|_{C^k(M,f)}$.  Changing the Morse function $f$ could possibly change $m_k$ as well.  In the above examples, the variations considered improved the bounds but did not reach the actual value of $b^\psi_k$ as given in \eqref{bStg}.  If we have chosen to work with a perfect Morse function on $S^2$, then by Corollary \ref{wperf}, we would have obtained the expected $b^\psi_k$ exactly. 
\end{rmk}

Finally, we show on $S^2$ using the same non-perfect Morse function $f=x^2+2y^2+3z^2$ how we can bound the Betti number as in \eqref{bmvIneq1} (Corollary \ref{bmv}) by considering different exact differential two-forms for $\psi$. 

\begin{ex}{\bf [Change in the bound of the Betti number from varying the exact two-form $\psi$.]} 
Let $\psi=d\alpha$ be an exact two-form on the sphere. We evaluate 
\begin{align*}
c(\psi)q=\displaystyle \sum_r\left(\int_{\overline{\CM(r,q)}}d\alpha\right) r =\displaystyle \sum_r\left(\int_{\partial {\overline{\CM(r,q)}}}\alpha\right) r\,.
\end{align*}  
We will use the notation where a gradient flow curve from $p_i^+$ to $p_j^-$ is labelled by $\gamma_{ij}^{+-}$ (and $\gamma_{ij}^{-+}$ represents the flow curve from $p_i^-$ to $p_j^+$).  Explicitly, we have
\begin{align*}
    \partial \overline{\CM(p_2^+,p_0^+)} &=\gamma_{21}^{+-}+\gamma_{10}^{-+}-\gamma_{21}^{++}-\gamma_{10}^{++}\, ,\\
    \partial \overline{\CM(p_2^-,p_0^+)} &=\gamma_{21}^{-+}+\gamma_{10}^{++}-\gamma_{21}^{--}-\gamma_{10}^{-+}\, , \\
    \partial \overline{\CM(p_2^+,p_0^-)} &=\gamma_{21}^{++}+\gamma_{10}^{+-}-\gamma_{21}^{+-}-\gamma_{10}^{--}\, , \\
    \partial \overline{\CM(p_2^-,p_0^-)} &=\gamma_{21}^{--}+\gamma_{10}^{--}-\gamma_{21}^{-+}-\gamma_{10}^{+-}\, .
\end{align*}
Let us further denote by $c_{21}^{+-}=\int_{\gamma_{21}^{+-}}\alpha$ and similarly for other line integral over $\alpha$.  The $c(d\alpha)$ map acting on $a_0^+ p_0^+ + a_0^- p_0^- \in C^0(S^2, f)\,$ then takes the following form. 
\begin{align*}
\begin{pmatrix}
    a_2^+ \\ a_2^-
\end{pmatrix}=
\begin{pmatrix}
 c_{21}^{+-}+c_{10}^{-+}-c_{21}^{++}-c_{10}^{++} & \mspace{10mu} c_{21}^{++}+c_{10}^{+-}-c_{21}^{+-}-c_{10}^{--} \\ 
 c_{21}^{-+}+c_{10}^{++}-c_{21}^{--}-c_{10}^{-+} & \mspace{10mu} c_{21}^{--}+c_{10}^{--}-c_{21}^{-+}-c_{10}^{+-}
\end{pmatrix}
\begin{pmatrix}
    a_0^+ \\ a_0^-
\end{pmatrix}\,.
\end{align*}
This $c(d\alpha)$ matrix has the following determinant: 
\begin{align*}
(c_{21}^{++}-c_{21}^{-+}+c_{21}^{--}-c_{21}^{+-})(c_{10}^{++}-c_{10}^{-+}+c_{10}^{--}-c_{10}^{+-})\,.
\end{align*}
Note that the first factor is the line integral of $\alpha$ over a meridian and the second factor the line integral of $\alpha$ over the equator. Thus, if we work with an one-form $\alpha$ such that both factors are non-zero (such one-forms are abundant, for instance, take $\alpha$ to be a positive one-form localized along $\gamma_{21}^{++}$ and $\gamma_{10}^{++}$), then $v_0=2$. From \eqref{bmvIneq1}, we thus find $b_1 \leq m_1-v_0= 0$, showing that the first Betti number of the two-sphere must be zero.
\end{ex}

\

\

\appendix
\section{Morse theory conventions and Leibniz rule}\label{AppA}
We describe here the conventions used to define the differential map $\del$ in the Morse cochain complex and also the orientations of the submanifolds which are integrated over in the $c(\psi)$ map of \eqref{cpsidef}.  The main aim of this appendix is to prove the following:
\begin{lem}[Leibniz Rule on forms in Morse cohomology]
Let $\psi \in \Omega^{\ell}(M)$ then 
\begin{align}\label{Alemma}
\del\, c(\psi)+(-1)^{\ell+1}c(\psi)\,\del=-c(d\psi)\,.
\end{align} 
\end{lem} 
This formula  appeared in Austin-Braam \cite{AB} and Viterbo \cite{Viterbo} though with ambiguous signs.  To set our conventions and prove the Lemma, we start with a brief background.

Let $\phi_t$ be the flow of the vector field $-\nabla f$.  For a critical point $r \in Crit(f)$, the stable $S_r$ and unstable $U_r$ submanifolds are defined to be
\begin{align*}
S_{r}=\{x \in M: \lim_{t \to \infty} \phi_t(x)=r\}\,, \qquad U_{r}=\{x \in M: \lim_{t \to -\infty} \phi_t(x)=r\}\,, 
\end{align*}
and the moduli spaces of gradient lines between two critical points, $q, r \in Crit(f)$, 
\begin{align*}
\CM(r, q)=S_q \cap U_r \,,\qquad \widetilde{\CM}(r,q)= \frac{S_q \cap U_r} {\{x \sim y: \phi_t(x)=y \text{ for some } t \in \mathbb{R}\}}.  
\end{align*}  
 We define the orientation of the moduli spaces similar to that in Austin-Braam \cite{AB}*{Section 2.2}.  For an oriented manifold $M$, we first specify an orientation for either the stable submanifolds, or equivalently, the unstable ones.  The orientation of one type determines the other by the relation
\begin{align}\label{SUMo}
[S_r][U_r] = [M]\,.
\end{align}
The orientation of the moduli space is then just the orientation of the transversal intersection which can be expressed as
\begin{align}
[\CM(r,q)]&=[U_r][M]^{-1}[S_q] = [U_r][U_q]^{-1}\,.\label{Mrq}
\end{align}
We will also take as convention 
\begin{align}
[\CM(r,q)]&=[\widetilde{\CM}(r,q)][\nabla f]\,.
\label{tMrq}
\end{align}

In the special case when $ind(r) = ind(q)+1$, $\CM(r,q)$ is an oriented one-dimensional submanifold of gradient flow lines and $\widetilde{\CM}(r,q)$ is an oriented collection of points. Also, recall that the Morse differential is defined by $\del q =\sum\limits_r n(r,q)\,r$ where
\begin{align}
n(r,q)=\#\widetilde{\CM}(r,q)\,.
\end{align}  
It follows from \eqref{tMrq} that $n(r,q)$ is equal to the number of gradient lines flowing in the direction of $\nabla f$ minus the number flowing in the direction of $-\nabla f$.

As an example of why \eqref{Alemma} has the correct signs, we first prove the zero-form case with $\psi=h$, a function.
\begin{cor}
If $h \in C^\infty(M)$, then $-c(dh)=\del c(h)-c(h)\del$. 
\end{cor} 
\begin{proof}
Evaluating $c(dh)$  
by integrating over the gradient curves with orientation, we have
\begin{align*} 
    c(dh)q_k&=\sum_{r_{k+1}}\left( \int_{\overline{\CM(r_{k+1}, q_k)}} dh\right)r_{k+1}\\ 
    &=\sum_{r_{k+1}}\left(n(r_{k+1}, q_k) (h(r_{k+1})-h(q_k))\right)r_{k+1} \\
    &=\sum_{r_{k+1}} h(r_{r+1}) n(r_{k+1}, q_k)r_{k+1} -\sum_{r_{k+1}}n(r_{k+1}, q_k)h(q_k)r_{k+1} \\
    &=c(h)\del q_k -\del c(h) q_k = (c(h)\del -\del c(h))q_k 
\end{align*} 
where  $c(h)q_k=(\int_{\overline{\CM(q_k, q_k)}}h)q_k = h(q_k)q_k$. Thus, having taken into account our orientation convention, we find that $-c(dh)=\del c(h)
-c(h)\del\,.$
\end{proof}
To prove \eqref{Alemma} in general, we re-express the right-hand side by Stokes' theorem
\[c(d\psi)q_k=\sum_{r_{k+\ell+1}}\left(\int_{ \overline{\CM(r_{k+\ell+1}, q_k)}}d\psi\right) r_{k+\ell+1}=\sum_{r_{k+\ell+1}}\left(\int_{\del \CM(r_{k+\ell+1},q_k)}\psi\right) r_{k+\ell+1}\,.\]
The relevant components of $\del \CM(r_{k+\ell+1}, q_k)$ for integrating $\psi$ consists of
\begin{align*}
 \left\{\bigcup_{p_{k+\ell}} \CM(p_{k+\ell}, q_k) \times \widetilde{\CM}(r_{k+\ell+1}, p_{k+\ell})\right\} 
   \;\bigcup\; \left\{\bigcup_{p_{k+1}} \CM(r_{k+\ell+1}, p_{k+1}) \times \widetilde{\CM}(p_{k+1}, q_k) \right\}\,. 
\end{align*}
This implies up to signs 
\begin{align}
    &c(d\psi)q_k \nonumber\\
    &=\!\!\sum_{r_{k+\ell+1}}\left[\pm\sum_{p_{k+\ell}} \int_{\overline{\CM(p_{k+\ell}, q_k)}\times \widetilde{\CM}(r_{k+\ell+1}, p_{k+\ell})}\!\!\psi\,\,  
    \pm\,\,\sum_{p_{k+1} }
    \int_{\overline{\CM(r_{k+\ell+1}, p_{k+1})}\times \widetilde{\CM}(p_{k+1}, q_{k})}\psi\right]r_{k+\ell+1}\nonumber\\
&=\!\!\!\!\sum_{r_{k+\ell+1}}\!\!\!\left[\pm\sum_{p_{k+\ell}} \left(\int_{\overline{\CM(p_{k+\ell}, q_k)}}\! \psi \right)\!n(r_{k+\ell+1}, p_{k+\ell})
    \pm\sum_{p_{k+1}} \!\!n(p_{k+1}, q_k)\!\!\left(\int_{\overline{\CM(r_{k+\ell+1}, p_{k+1})}}\!\psi \right)\!\right]\!r_{k+\ell+1}\nonumber\\
    & =\pm\del c(\psi)q_k\pm c(\psi)\del q_k \label{SemiStokes}
\end{align}
To fix the signs, we will proceed in two steps. First, we make a choice of the orientation of the stable and unstable manifolds at the critical points $\{q_k, p_{k+1}, p_{k+l}, r_{k+l+1}\}$. By \eqref{Mrq}, this determines the orientation of the various moduli spaces that arise in the Stokes' theorem calculation above. Then in step two, we compare the orientation of the relevant boundary components, $\CM(p_{k+\ell}, q_k)\times\widetilde{\CM}(r_{k+l+1}, p_{k+l})$ and 
$\CM(r_{k+\ell+1},p_{k+\ell})\times {\widetilde \CM}(p_{k+1},q_k)$, 
with the orientation needed to satisfy Stokes' theorem.  The relative difference in the orientations will determine the signs in \eqref{SemiStokes}. 

\

\textit{Step 1:
Computing the orientation of the moduli spaces.}

By \eqref{Mrq}, the orientation of a moduli space $\CM(r,q)$ can be determined by the orientation of the unstable submanifolds $U_r$ and $U_q$.  Hence, we will write below our choice for the orientation for the relevant unstable submanifolds explicitly.  (The orientation of the stable submanifolds of a critical point are then fixed by \eqref{SUMo}.)  Similar to \cite{AB}*{Section 2.2},
we will express the orientations in terms of orthonormal frame vectors grouped together by Clifford multiplication.

Let $e_1, \ldots, e_k$ be an orthonormal set of frame vectors that are shared by both $U_{q_k}$ and $U_{r_{k+\ell+1}}$.   Let $e_{k+1}, \dots, e_{k+\ell+1}$ be the additional frame vectors in $U_{r_{k+\ell+1}}$ defined such that they point in the direction away from $q_k$ towards $r_{k+l+1}$, i.e. in the direction of $\nabla f$. Then, for $p_{k+1}$, there is a vector $e_{i_{p_{k+1}}}$ that points along the gradient curve $\CM(p_{k+1}, q_k)$ from $q_k$ to $p_{k+1}$, and for $p_{k+\ell}$, there is a vector $e_{i_{p_{k+\ell}}}$ that points along the gradient curve $\CM(r_{k+l+1}, p_{k+l})$ from $p_{k+\ell}$ to $r_{k+\ell+1}$.  Note both $e_{i_{p_{k+1}}}$ and $e_{i_{p_{k+\ell}}}$ are defined to point in the direction of $\nabla f$. See Figure \ref{Afigure} below.
\begin{figure}[h]
\centering
\includegraphics[scale=.6]{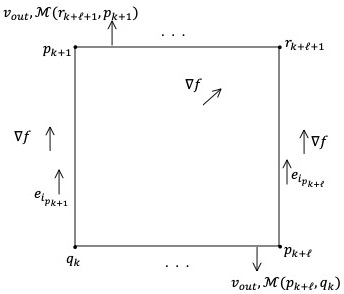} \\
\caption{$\CM(r_{k+\ell+1}, q_k)$ with orientations.}
\label{Afigure}
\end{figure}

Our choice for the orientation of the relevant unstable submanifolds are
\begin{align*}
&[U_{q_k}]=e_k\ldots e_1\,, \qquad\qquad\qquad  &&[U_{p_{k+\ell}}]=e_{k+\ell+1}\ldots \widehat{e_{i_{p_{k+\ell}}}}\ldots e_k\ldots e_{1}\,,\\ 
&[U_{p_{k+1}}]=e_{i_{p_{k+1}}}e_k\ldots e_1\,, \quad &&[U_{r_{k+\ell+1}}]=e_{k+\ell+1}\ldots e_k\ldots e_1\,.
\end{align*} 
 
Then by \eqref{Mrq}, $[\CM(r,q)]=[U_r][U_q]^{-1}$, we find the orientations of the moduli spaces: 
\begin{align}
[\CM(r_{k+\ell+1},q_k)]
&=(e_{k+\ell+1}\ldots e_k\ldots e_{1})(e_1\ldots e_k) =e_{k+\ell+1}\ldots e_{k+1}\,,\label{MrqO}\\
[\CM(p_{k+\ell}, q_k)]
&=(e_{k+\ell+1}\ldots \widehat{e_{i_{p_{k+\ell}}}}\ldots e_k\ldots e_{1})(e_1\ldots e_k)=e_{k+\ell+1}\ldots \widehat{e_{i_{p_{k+\ell}}}}\ldots e_{k+1}\,,\nonumber
\\
[\CM(r_{k+\ell+1}, p_{k+1})]
&=(e_{k+\ell+1}\ldots e_k\ldots e_{1})(e_1\ldots e_k e_{i_{p_{k+1}}})\nonumber\\
&=(-1)^{i_{p_{k+1}}-k-1}\,e_{k+\ell+1}\ldots \widehat{e_{i_{p_{k+1}}}}\ldots e_{k+1}\,.\nonumber
\end{align}
And by \eqref{tMrq}, we also have
\begin{align*}
[\widetilde{\CM}(r_{k+\ell+1}, p_{k+\ell})]&=[\CM(r_{k+\ell+1}, p_{k+\ell})][\nabla f]^{-1}\\
&=(e_{k+\ell+1}\ldots e_k \ldots e_1)(e_1\ldots e_{k}\ldots \widehat{e_{i_{p_{k+\ell}}}}\ldots e_{k+\ell+1})(e_{i_{p_{k+\ell}}})\\
&=(-1)^{k+\ell+1-i_{p_{k+\ell}}}\,,\\
[\widetilde{\CM}(p_{k+1}, q_k)]&=[\CM(p_{k+1}, q_k)][\nabla f]^{-1}\\
&=(e_{i_{p_{k+1}}}e_k\ldots e_1)(e_1\ldots e_{k})(e_{i_{p_{k+1}}})=1\,.
\end{align*} 
Hence, we find
\begin{align}
[\CM(p_{k+\ell}, q_k) \times \widetilde{\CM}(r_{k+\ell+1},p_{k+\ell})]&=(-1)^{k+\ell+1-i_{p_{k+\ell}}}\,e_{k+\ell+1}\ldots \widehat{e_{i_{p_{k+\ell}}}}\ldots e_{k+1}\,,\label{MM1}\\
[\CM(r_{k+\ell+1}, p_{k+1}) \times \widetilde{\CM}(p_{k+1},q_k)]&=(-1)^{i_{p_{k+1}}-k-1}\,e_{k+\ell+1}\ldots \widehat{e_{i_{p_{k+1}}}}\ldots e_{k+1}\,.\label{MM2}
\end{align}

\textit{Step 2: Orientation of the boundary components, $\CM(p_{k+\ell}, q_k) \times \widetilde{\CM}(r_{k+\ell+1},p_{k+\ell})$ and $\CM(r_{k+\ell+1}, p_{k+1}) \times \widetilde{\CM}(p_{k+1},q_k)$, as specified by Stokes' theorem.}

For a manifold $N$ with boundary $\del N$, Stokes' theorem holds only if the 
orientation of the boundary $\del N$ is chosen such that 
\begin{align}
[v_{out}][\del N] = [N]
\end{align}
where $v_{out}$ is the outward pointing normal on the boundary. 

For the boundary component $\CM(p_{k+\ell}, q_k) \times \widetilde{\CM}(r_{k+\ell+1},p_{k+\ell})$, the outward pointing normal at for instance $p_{k+\ell}$ can be expressed as (see Figure \ref{Afigure}) 
\[v_{out,\CM( p_{k+\ell},q_k)}=-e_{i_{p_{k+\ell}}}+ \sum_{k+j\neq i_{p_{k+\ell}}} a_je_{k+j}\,.\]
Therefore, the specified orientation from Stokes' theorem (denoted with a subscript `$S$') is
\begin{align}
[\CM(p_{k+\ell}, q_k) \times \widetilde{\CM}(r_{k+\ell+1},p_{k+\ell})]_S &= [v_{out, \CM( p_{k+\ell},q_k)}]^{-1}[\CM(r_{k+\ell+1}, q_k)]\nonumber\\
&=(-e_{i_{p_{k+\ell}}})(e_{k+\ell+1}\ldots e_{k+1})\nonumber\\
&=(-1)^{k+\ell+i_{p_{k+\ell}}}\,e_{k+\ell+1}\ldots \widehat{e_{i_{p_{k+\ell}}}}\ldots e_{k+1}\nonumber\\
&=-[\CM(p_{k+\ell}, q_k) \times \widetilde{\CM}(r_{k+\ell+1},p_{k+\ell})]\label{Mpkls}
\end{align}
having used \eqref{MrqO} in the first line and \eqref{MM1} in the last line.  

Similarly, for the boundary component $\CM(r_{k+\ell+1}, p_{k+1}) \times \widetilde{\CM}(p_{k+1},q_k)$, the outward pointing normal at for instance $p_{k+1}$ can be expressed as (see Figure \ref{Afigure})
\[v_{out, \CM( r_{k+\ell+1},p_{k+1})}=e_{i_{p_{k+1}}} + \sum_{ k+j \neq i_{p_{k+1}}}a_je_{k+j}\,.\]
This gives for the specified orientation from Stokes' theorem
\begin{align}
[\CM(r_{k+\ell+1}, p_{k+1}) \times \widetilde{\CM}(p_{k+1},q_k)]_S&=[v_{out, \CM( r_{k+\ell+1},p_{k+1})}]^{-1}[\CM(r_{k+\ell+1}, q_k)]\nonumber\\
&=( e_{i_{p_{k+1}}})(e_{k+\ell+1}\ldots e_{k+1})\nonumber\\
&=(-1)^{k+\ell+1-i_{p_{k+1}}}\,e_{k+\ell+1}\ldots \widehat{e_{i_{p_{k+1}}}}\ldots e_{k+1}\nonumber\\
&=(-1)^\ell [\CM(r_{k+\ell+1}, p_{k+1}) \times \widetilde{\CM}(p_{k+1},q_k)] \label{Mrkls}
\end{align}
having used \eqref{MrqO} in the first line and \eqref{MM2} in the last line.

\

Finally, with \eqref{Mpkls}-\eqref{Mrkls} and matching up with the corresponding terms in \eqref{SemiStokes}, we have 
\begin{align*}  
c(d\psi)q_k&=\sum_{r_{k+\ell+1}}\Bigg[-\sum_{p_{k+\ell}}\left(\int_{\overline{\CM(p_{k+\ell}, q_k)}}\psi\right)n(r_{k+\ell+1},p_{k+\ell}) \\
&\qquad\qquad\qquad+\sum_{p_{k+1} }(-1)^{\ell}n(p_{k+1},q_k)\left(\int_{\overline{\CM(r_{k+\ell+1}, p_{k+1})}}\psi\right)\Bigg]r_{k+\ell+1} \\ 
    &=-\del c(\psi)q_k+(-1)^{\ell}c(\psi)\del q_k
\end{align*}
or equivalently, $-c(d\psi)=\del c(\psi)+(-1)^{\ell+1}c(\psi)\del$.

\section{Cochain complexes of a chain map and their relations}\label{AppB}

We here review some relations between cochain complexes that arise from a chain map. For a reference, see \cite{Weibel}.

Let $\vp: (B, \db) \to (A, \da)$ be a degree $\ell$ chain map between two cochain complexes, 
\begin{equation}
\begin{tikzcd}
\ldots \arrow[r, "\da"]  & A^{n-1} \arrow[r,"\da"] & A^n \arrow[r,"\da"]  & A^{n+1}\arrow[r, "\da"] & \ldots\\
\ldots \arrow[r, "\db"] & B^{n-\ell-1} \arrow[r, "\db"] \arrow[u,swap,"\,\vp_{}"] & B^{n-\ell} \arrow[r, "\db"]\arrow[u,swap,"\,\vp_{}"]  & B^{n-\ell+1} \arrow[r, "\db"]\arrow[u, swap, "\,\vp_{}"]& \dots
\end{tikzcd}
\end{equation}
that is, the map $\vp:  B^n \to A^{n+\ell}$ satisfy the chain map condition 
\begin{align}\label{vpch}
\vp\, \db = \da\, \vp \,. 
\end{align}

Associated to such a map are the following cochain complexes.
\begin{itemize}
\item[(1)] The \textit{kernel} complex $(\ker \vp, \db)$. This is a subcomplex of $(B, \db)$    where $\ker^n\!\vp = \ker \vp \;\cap \,B^n$.  To see that $\db: \ker^n\!\vp \to \ker^{n+1}\!\vp\,$, consider an element $b_n \in \ker^n\!\vp$, i.e. $\vp\, b_n =0$.  By \eqref{vpch}, we have  $\vp (\db b_n) = \da (\vp\, b_n) = 0$; hence,  $\db b_n \in \ker^{n+1}\!\vp\,$. 
\item[(2)] The \textit{image} complex $(\im \vp, \da)$.  This is a subcomplex of $(A, \da)$ where $\im^n\!\vp = \im\vp \; \cap \, A^n$.  Specifically, if $a_n \in \im^n\!\vp$, then there exists an $b_{n-\ell}\in B^{n-\ell}$ such that $a_n = \vp\, b_{n-\ell}$.  Again, it follows directly from \eqref{vpch} that $\da: \im^n\!\vp \to \im^{n+1}\!\vp$.   
\item[(3)] The \textit{cokernel} complex $(\coker \vp, \dap)$.  This is also a subcomplex of $(A,\da)$ where $\coker^n\!\vp = A^n / \im \vp_{}$ and $\dap = \pi \,\da$ is the composition of $\da$ with the quotient map $\pi: A^n \to \coker^n\!\vp$.   To denote elements of the cokernel complex, we shall use a bracket, i.e. $[a_n]\!:=\{ a_n + \vp\, b_{n-\ell}\, |\, b_{n-\ell}\in B^{n-\ell}\} \in \coker^n\!\vp$. Note that $\dap[a_n]=[\da a_n]$, and therefore, $\dap\,\dap=0$.
\item[(4)] The (mapping) \textit{cone} complex $(\tC(\vp), \dc)$, the main focus of this paper, involves both $(B, \db)$ and $(A, \da)$.  Here,
\begin{align*}
\tC^n(\vp) = A^n \oplus B^{n-\ell+1} \, , \qquad \dc=\begin{pmatrix} \da & \vp \\ 0 & - \db\end{pmatrix}\,.
\end{align*}
with $\dc: \tC^n(\vp) \to \tC^{n+1}(\vp)$. Note that the chain map relation \eqref{vpch} ensures that $\dc\, \dc =0\,$.  \\
\end{itemize}
Each of the above cochain complexes results in a cohomology, denoted by $H^n(\ker \vp)$, $H^n(\im \vp)$, $H^n(\coker \vp)$, and $H^n(\tC(\vp))$, respectively.   We are interested in the relations amongst these cohomologies and also their relations with $H^n(A)$ and $H^n(B)$.  A first basic relation used throughout the paper follows from the following short exact sequence of cochain complexes 
\begin{equation*}
\begin{tikzcd}
0 \arrow[r]  & (A^n, d_A) \arrow[r,"\iota_A"]  & (\tC^n(\vp), d_C)  \arrow[r,"\pi_B"]  & (B^{n-\ell+1}, -d_B) \arrow[r] & 0
\end{tikzcd}
\end{equation*}
where the chain map $\iota_A$ is the inclusion into the first component of $\tC(\vp)$ and $\pi_B$ is the projection of the second component.  The short exact sequence gives the long exact sequence 
\begin{equation}
\begin{tikzcd} 
\cdots\ H^{n-\ell}(B) \arrow[r, "\vp"] & H^{n}(A) \arrow[r, 
"\iota_A"] & H^{n}(\tC(\vp)) \arrow[r, "\pi_B"] & H^{n-\ell+1}(B) \arrow[r, "\vp"] & H^{n+1}(A)\ \cdots
\end{tikzcd}
\end{equation}
which implies the following:
\begin{lem}\label{hckkba}
Given a degree $\ell$ chain map $\vp: (B, \db) \to (A, \da)$ between two cochain complexes, the resulting cone cohomology has the following relation: 
\begin{align*}
H^{n}(\tC(\vp))\cong \coker(\,\vp\!: H^{n-\ell}(B)\to H^n(A)\,)\,\oplus\,\, \ker(\,\vp\!: H^{n-\ell+1}(B)\to H^{n+1}(A)\,)\,.
\end{align*} 
\end{lem}
To relate the other cohomologies, it is useful to introduce another cone complex defined by the inclusion map $\iota: \im^n\!\vp \to A^n$, which is a degree $\ell=0$ map.  We shall denote this cone complex with a tilde:
\begin{align*}
\ttC^n\!(\iota) = A^n \oplus \im^{n+1}\!\vp\, , \qquad \dcc=\begin{pmatrix} \da & \iota \\ 0 & - \da\end{pmatrix}\,.
\end{align*}
Of note, the cohomology of this complex, $H^n(\ttC(\iota))$ is isomorphic to $H^n(\coker\vp)$.

\begin{lem}\label{cokernelH}
The map $\pio: \ttC^n\!(\iota) \rightarrow \coker^n\!\vp$ given by $\pio\begin{pmatrix} a \\ \tila \end{pmatrix}
= \pi \, a$, where $\pi: A^n\to \coker^n\!\vp$,
induces an isomorphism on cohomology: $H^n(\ttC(\iota))\cong H^n(\coker\vp)\,$.
\end{lem}
\begin{proof}
That the $\pio$ map is a chain map follows straightforwardly from the definition.   To prove the isomorphism, we will show that $\pio: H^n(\ttC(\iota))\rightarrow H^n(\coker\vp)$ is bijective. 
 
Let $[a] \in \coker^n\!\vp$.  
To show surjectivity, assume $[a]\in H^n(\coker\vp)$, that is, $[a]$ is closed under $\dap=\pi\,\da$, or equivalently, that the representative $a\in A^n$ satisfies 
\begin{align}\label{surjeq}
\da a + \vp\, b = 0
\end{align}
for some $b\in B^{n-\ell+1}$.  Now let $\tila = \vp\, b$.  Then since $\da \da =0\,$,  \eqref{surjeq} implies $\da \tila =0$.  Therefore, the pair $\begin{pmatrix} a \\ \tila \end{pmatrix}$ is $\dcc$-closed, i.e.~it is an element of $H^n(\ttC(\iota))$, and moreover, $\pio: \begin{pmatrix} a \\ \tila \end{pmatrix} \rightarrow [a]$ as desired.  

To show that $\pio$ is also injective, let now $[a] = \pi \da [a']$ representing the trivial class in $H^n(\coker\vp)$.  This implies that $a = \da a' + \vp\, b'$ for some $b' \in B^{n-\ell}$.  But this also means, 
\begin{align*}
\pio\, \left\{ \dcc\begin{pmatrix} a' \\ \vp\,b' \end{pmatrix}\right\} = \pio \begin{pmatrix} \da a' + \vp \, b' \\ - \da\, \vp\,b' \end{pmatrix}  = \pi (\da a' + \vp\,b') = [a]\,.
\end{align*}
Hence, $\pio$ maps trivial class to trivial class.
\end{proof}
Now applying Lemma \ref{hckkba} to $H^n(\ttC(\iota))$ and using Lemma \ref{cokernelH}, we find the following:
\begin{lem} 
For the cohomology of the cokernel complex, we have
\begin{align*}
H^n(\coker\vp)\cong \coker(\,\iota\!: H^{n}(\im\vp)\to H^n(A)\,)\,\oplus\,\, \ker(\,\iota\!: H^{n+1}(\im\vp)\to H^{n+1}(A)\,)\,,
\end{align*}
where $\iota: \im^n\!\vp \to A^n$ is the inclusion map.
\end{lem}
Finally, we give a relation that links $H(\tC(\vp))$ with $H(\ker \vp)$ and $H(\coker \vp)$.  At the cochain level, we can write down the following short exact sequence of cochain complexes: 
\begin{equation}\label{sestc}
\begin{tikzcd}
0 \arrow[r]  &  (\ker^{n-\ell+1}\!\vp, \db)  \arrow[r,"\itwo"] & (\tC^n(\vp), \dc) \arrow[r,"\vpt"]  & (\ttC^n\!(\iota), \dcc)\arrow[r, ] & 0
\end{tikzcd}
\end{equation}
where the maps $\itwo$ and $\vpt$ are defined by
\begin{alignat*}{2}
\itwo:  \ker^{n-\ell+1} \!\vp &\longrightarrow  \tC^n(\vp).  \qquad \qquad &\vpt: \tC^n(\vp) & \longrightarrow \ttC^n\!(\iota)\\
 b \qquad &\longmapsto  \quad\begin{pmatrix} 0 \\ b \end{pmatrix}  & \begin{pmatrix} a \\ b \end{pmatrix} \quad & \longmapsto ~ \begin{pmatrix} a \\ \vp \, b \end{pmatrix} 
\end{alignat*}
The short exact sequence \eqref{sestc} implies the following long exact sequence of cohomology:
\begin{align*}
\ldots \xrightarrow{\delta} H^{n-\ell+1}(\ker\vp) \xrightarrow[]{\iota_2} H^{n}(\tC(\vp))\xrightarrow[]{\vpt} H^{n}(\ttC(\iota)) \xrightarrow[]{\delta} H^{n-\ell+2}(\ker\vp) \xrightarrow[]{\iota_2} \ldots 
\end{align*}
where the connecting homomorphism $\delta$ can be obtained by standard diagram chasing.  Now using Lemma \ref{cokernelH} to replace $H^{n}(\ttC(\iota))$ by $H^{n}(\coker\vp)$, we have derived the below long exact sequence.
\begin{lem}
Let $\vp: (B, \db) \to (A, \da)$ be a degree $\ell$ chain map between cochain complexes.  Then there exists a connecting homomorphism $\delta'$ such that
\begin{align}\label{leskca}
\dots 
\xrightarrow[]{\delta'} H^{n-\ell+1}(\ker\vp) \xrightarrow[]{\iota_2} H^{n}(\tC(\vp))\xrightarrow[]{\pio\,\circ\,\vpt} H^{n}(\coker\vp) \xrightarrow[]{\delta'} H^{n-\ell+2}(\ker\vp) \xrightarrow[]{\iota_2}
\dots 
\end{align} 
is a long exact sequence.
\end{lem}

\section{Differential graded algebra for mapping cones of differential forms}\label{AppC}

We can give a differential graded algebra structure for the mapping cone of the map $\psi\w$ where $\psi\in \Om^\ell(M)$ is a $d$-closed form.  Our description will differ slightly depending on whether $\ell$ is even or odd. 
  
Consider first the case when $\ell$ is even.  A simple way to motivate the algebra structure on $\tC^\bullet(\psi)=\Om^\bullet(M) \oplus \Om^{\bullet-\ell+1}$ is to introduce a formal object, $\theta$, that acts like a differential $(\ell-1)$-form with two defining properties: (i) $d\theta = \psi$; (ii) $\theta \w \theta = 0$.  Making use of $\theta$, we can express $\tC^k(\psi)= \Om^k(M) \oplus \theta\w \Om^{k-\ell+1}(M)$ which now has the same total degree grading on both components.  Moreover, the cone differential $d_C$ can be interpreted simply as an exterior derivative: 
\begin{align}\label{C1}
d (\Om^k \oplus \theta\w \Om^{k-\ell+1}) = (d \Om^k + \psi \w \Om^{k-\ell+1}) \oplus \theta \w (-d \Om^{k-\ell+1}) = d_C \tC(\psi)\,.
\end{align} 
We can thus treat $\tC^k(\psi)$ formally as a differential form space and define the product operation on cone forms by means of the standard wedge product:
\begin{align}
\tC^j(\psi) \times \tC^k(\psi) :&= (\Om^j  \oplus \theta\w \Om^{j+\ell-1}) \w (\Om^k \oplus \theta \w \Om^{k+\ell-1})\nonumber\\
& = (\Om^j \w \Om^k) \oplus \theta\w \left[\Om^{j+\ell-1}\w \Om^k+(-1)^{j}\Om^j\w \Om^{k-\ell+1}\right]. \label{C2}
\end{align} 
Clearly, this product satisfies the Leibniz rule, and hence, $(\tC^\bullet(\psi), d_C, \times)$ is a differential graded algebra that is dependent on $\psi$.

Now, we consider the case when $\ell$ is odd. The above product formula (\ref{C2}) does not generalize directly as $\theta$ is now formally an even degree form.  Requiring $\theta\w\theta=0$ together with $d\theta=\psi$ would imply an additional consistency constraint of $\frac{1}{2}d(\theta\w\theta)= \theta\w \psi=0,$ which is unnatural if imposed on the second component $\theta\w \Om^{k-\ell+1}$.  Instead, we can build in this additional constraint by considering a modified cone space.  Of note, since $\psi$ is a form of odd degree, we have the property $\psi \w \psi = 0\,$.
It is therefore natural to consider the quotient $\pi: \Om^k \to \Om^k/\im^k \psi = \coker^k \psi\,$ and its associated chain complex $(\coker^\bullet\psi, d^\pi)$ where  $d^\pi=\pi\,d$ (see Appendix \ref{AppB}) as it gives a more refined map $\psi \w : \coker^k\psi \to \Om^{k+\ell}(M)$.  This map results in the following cone complex:
\begin{align}\label{C3}
\wtC^k\!(\psi) = \Om^k \oplus \coker^{k-\ell+1}\psi\, , \qquad \dcw=\begin{pmatrix} d & \psi\w \\ 0 &  d^\pi\end{pmatrix}\,.
\end{align}
Importantly, we can write $\wtC^k\!(\psi)= \Om^k \oplus \theta \w \pi(\Om^{k-\ell+1})$, where $\theta$ is again the formal $(\ell-1)$-degree form with properties $d\theta = \psi$ and $\theta \w \theta = 0$.  With the presence of the quotient operator $\pi$ on the right side of $\theta$, the formerly concerning condition $\theta \w \psi\w \Om=0$ becomes now $\theta \w \pi(\psi\w\Om )=0$ which holds trivially.  Moreover, the modified cone differential $\dcw$ continues to act as an exterior derivative,
\begin{align}\label{C4}
d \left(\Om^k \oplus \theta\w\pi (\Om^{k-\ell+1})\right) = (d \Om^k + \psi \w \Om^{k-\ell+1}) \oplus \theta \w\pi(d \Om^{k-\ell+1}) = \dcw \wtC^k(\psi)\,.
\end{align} 
Thus, $\wtC^k(\psi)$ can still be treated as a single differential form space and we define the product operation on $\wtC^\bullet(\psi)$ similar to \eqref{C2}:
\begin{align}\label{C5}
\wtC^j(\psi) \times \wtC^k(\psi) :
& = (\Om^j \w \Om^k) \oplus \theta\w \pi\left(\Om^{j+\ell-1}\w \Om^k+\Om^j\w \Om^{k-\ell+1}\right).
\end{align} 
Altogether, $(\wtC^\bullet(\psi), \dcw, \times)$ is a differential graded algebra for $\psi$, an odd degree, $d$-closed form.

\


\

\vskip 1cm

\noindent
{Department of Mathematics, University of California, Riverside, CA 92521, USA}\\
{\it Email address:}~{\tt dclausen@ucr.edu}
\vskip .5 cm
\noindent
{Department of Mathematics, Washington University, St. Louis, MO 63130, USA}\\
{\it Email address:}~{\tt xtang@math.wustl.edu}
\vskip .5 cm
\noindent
{Department of Mathematics, University of California, Irvine, CA 92697, USA}\\
{\it Email address:}~{\tt lstseng@uci.edu}

\end{document}